\documentclass[12pt]{article}
\font\smallit=cmti10

\makeatletter

\renewcommand\section{\@startsection {section}{1}{\z@}
{-30pt \@plus -1ex \@minus -.2ex}
{2.3ex \@plus.2ex}
{\normalfont\normalsize\bfseries}}

\renewcommand\subsection{\@startsection{subsection}{2}{\z@}
{-3.25ex\@plus -1ex \@minus -.2ex}
{1.5ex \@plus .2ex}
{\normalfont\normalsize\bfseries}}

\renewcommand{\@seccntformat}[1]{\csname the#1\endcsname. }

\makeatother

\newtheorem{thm}{Theorem}[section]
\newtheorem{prop}[thm]{Proposition}
\newtheorem{lemma}[thm]{Lemma}
\newtheorem{remark}[thm]{Remark}

\newenvironment{proof}{\par\noindent\textsc{Proof}}{\medskip}
\def\qed{\relax\ifmmode\hskip2em \Box\vspace{-7pt}
 \else\unskip\nobreak\hskip1em $\Box$\fi}

\usepackage{amsmath,amsfonts,amssymb,latexsym}

\usepackage{enumerate} %example \begin{enumerate}[1)]

\newcommand{\n}[2]{\put(#1,#2){\circle{1.10}}} %white circle
\newcommand{\y}[2]{\put(#1,#2){\circle*{1.10}}} %grey circle
\usepackage{float}
\newfloat{d2}{htb}{der}[section]
\floatname{d2}{Table}

\begin{document}
\vskip 40pt

\begin{center}
  \uppercase{\bf Small and Large Weights in Steinhaus Triangles}\\
   \vskip 20pt
  {\bf Josep M. Brunat and
    Montserrat Maureso} \\
  {\smallit Departament de Matem\`atica Aplicada, Universitat
    Polit\`ecnica  de Catalunya, Barcelona, Catalunya}\\
  {\tt josep.m.brunat@upc.edu,
    montserrat.maureso@upc.edu}\\ %(optional)
  \vskip 10pt
\end{center}
\vskip 30pt

\centerline{\bf Abstract}

\noindent
Let $\{0=w_0<w_1<w_2<\ldots<w_m\}$ be the set of weights of binary Steinhaus triangles of size $n$,
and let $W_i$ be the set of sequences in $\mathbb{F}_2^n$ that generate triangles
of weight  $w_i$.
We obtain  the values of $w_i$ and
the corresponding sets $W_i$ for $i\in\{1,2,3,m\}$, and partial results about $w_{m-1}$ and $W_{m-1}$.

\baselineskip=15pt
\vskip 30pt

%---------------------------------
\section{Introduction}

An extended abstract of this paper has been publised under the title \emph{Extreme Weights in Steinhaus Triangles}~\cite{BrMa0}.

Let $\mathbb{F}_2$ be the field of order $2$ and
$\mathbf{x}=(x_0,\,\ldots,\,x_{n-1})\in\mathbb{F}_2^n$ a binary sequence of
length $n$. We shall often write sequences $\mathbf{x}$ as  words $\mathbf{x}=x_0x_1\ldots x_{n-1}$.
The \emph{derivative} of $\mathbf{x}$ is the sequence
$\partial\mathbf{x}=(x_0+x_1,\,x_1+x_2,\,\ldots,\, x_{n-2}+x_{n-1})$.
It is clear that the mapping $\partial\colon\mathbb{F}_2^n\rightarrow\mathbb{F}_2^{n-1}$ is linear
and exhaustive, and that its kernel is formed by the two sequences
$\mathbf{0}=(0,0,\ldots,0)$ and $\mathbf{1}=(1,1,\ldots,1)$.
If
$\mathbf{y}=(y_0,\ldots,y_{n-2})\in\mathbb{F}_2^{n-1}$, then the two sequences
$\mathbf{x}\in\mathbb{F}_2^n$ such that $\partial \mathbf{x}=\mathbf{y}$ are
 $$
 \mathbf{x}_0=(0,\, y_0,\, y_0+y_1,\, y_0+y_1+y_2,\, \ldots,\, y_0+y_1+y_2+\ldots+y_{n-2})
 \quad \mbox{and}\quad \mathbf{x}_1=\mathbf{1}+\mathbf{x}_0,
 $$
and they are called the \emph{primitives} of $\mathbf{y}$.

We define $\partial^0\mathbf{x}=\mathbf{x}$,
$\partial^1\mathbf{x}=\partial\mathbf{x}$ and, for $2\le i\le n-1$,
$\partial^i\mathbf{x}=\partial\partial^{i-1}\mathbf{x}$.  The
\emph{Steinhaus triangle} of the sequence $\mathbf{x}$ is the sequence
$T(\mathbf{x})$ formed by $\mathbf{x}$ and its derivatives:
$T(\mathbf{x})=(\mathbf{x},\,\partial\mathbf{x},\,
\ldots,\,\partial^{n-1}\mathbf{x})$.  For $i\in\{0,\ldots,n-1\}$, the
component $\partial^i\mathbf{x}$ of $T(\mathbf{x})$ is the $i$-th
\emph{row} of the triangle.  Figure~\ref{fexemple1} represents
$T(\mathbf{x})$ for the sequence $\mathbf{x}=(0,\, 0,\, 0,\, 1,\, 0,\,
0,\, 1)$. The grey and white circles represent ones and zeros,
respectively; the first row corresponds to $\mathbf{x}$ and the
following rows to the iterated derivatives. Each entry of the triangle
is the binary sum of the two values immediately above it.
\begin{figure}[htb]
\centering
\setlength{\unitlength}{2.35mm}
\begin{picture}(7,9)
\n{0.5}{7.5} \n{1.5}{7.5} \n{2.5}{7.5} \y{3.5}{7.5} \n{4.5}{7.5} \n{5.5}{7.5} \y{6.5}{7.5}
\n{1.0}{6.5} \n{2.0}{6.5} \y{3.0}{6.5} \y{4.0}{6.5} \n{5.0}{6.5} \y{6.0}{6.5}
\n{1.5}{5.5} \y{2.5}{5.5} \n{3.5}{5.5} \y{4.5}{5.5} \y{5.5}{5.5}
\y{2.0}{4.5} \y{3.0}{4.5} \y{4.0}{4.5} \n{5.0}{4.5}
\n{2.5}{3.5} \n{3.5}{3.5} \y{4.5}{3.5}
\n{3.0}{2.5} \y{4.0}{2.5}
\y{3.5}{1.5}
\end{picture}
\caption{Steinhaus triangle  $T(\mathbf{x})$ of the sequence
  $\mathbf{x}=(0,\, 0,\, 0,\, 1,\, 0,\, 0,\, 1)$.}
\label{fexemple1}
\end{figure}
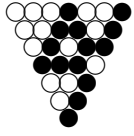
It is easily shown by induction that the $\ell$-th entry of the $j$-th row $\partial^j\mathbf{x}$ is
$$
(\partial^j\mathbf{x})_\ell=\sum_{k=0}^j{j\choose\ell}x_{j+\ell},\qquad (j\in\{0,\ldots,n-1\},\quad
\ell\in\{0,\ldots,n-1-j\}).
$$
If $\mathbf{x}\in\mathbb{F}_2^n$ and $k\ge 1$ is an integer, a \emph{subtriangle} of size $k$ of $T(\mathbf{x})$ is
the Steinhaus triangle generated by $k$ consecutive entries of a row $\partial^j\mathbf{x}$.

In 1958, H.~Steinhaus~\cite{Steinhaus} posed the question for which sequences
$\mathbf{x}\in\mathbb{F}_2^n$ the triangle $T(\mathbf{x})$ is
balanced, that is, $T(\mathbf{x})$ has as many zeroes as ones. He
observed that no sequence of length $n\equiv 1,2\ (\textrm{mod}\ 4)$
produces a balanced triangle, so the problem was to decide if they
exist for lengths $n\equiv 0,3 \ (\textrm{mod}\ 4)$.
H.~Harborth~\cite{Harborth} answered the question in the affirmative
by constructing examples of such sequences.  S.~Eliahou et al.
studied binary sequences generating balanced triangles with some
additional condition: sequences of length $n$, all of whose initial
segments of length $n-4t$ for $0\le t\le n/4$ generate balanced
triangles~\cite{ElHa}, symmetric and anti-symmetric
sequences~\cite{ElHa2}, and sequences with zero
sum~\cite{ElMaRe}. F.M. Malyshev and E.V.  Kutyreva~\cite{MaKu}
estimated the number of Steinhaus triangles (which they call Boolean
Pascal triangles) of sufficiently large size $n$ containing a given
number $\omega\le kn$ ($k>0$) of ones.  More recently, J.~Chappelon and S.~Eliahou~\cite{Chappelon,ChEl}
considered a generalization by J. C. Molluzzo~\cite{Molluzo} to sequences with entries in $\mathbb{Z}_m$,
with the condition that every element in $\mathbb{Z}_m$ has the same
multiplicity in the triangle. Steinhaus triangles appear in the
context of cellular automate, see A.~Barb\'e~\cite{Barbe2,Barbe3,Barbe} and J. Chappelon~\cite{Chappelon4,Chappelon5}.
In this context, A. Barb\'e~\cite{Barbe} has
studied some properties related to symmetries. In~\cite{BrMa} we characterized
Steinhaus triangles with rotational and dihedral symmetry.

The \emph{weight} of a  sequence $\mathbf{x}=(x_0,\ldots,x_{n-1})$ is the number $|\mathbf{x}|$ of ones that contains:
$|\mathbf{x}|=\#\{i\in\{0,\ldots,n-1\}: x_i=1\}$. The  \emph{weight}
of the triangle $T(\mathbf{x})$, denoted by $|T(\mathbf{x})|$ is the sum of the weights of its rows:
$$
|T(\mathbf{x})|=\sum_{i=0}^{n-1}|\partial^i\mathbf{x}|.
$$

The set $S(n)$ of Steinhaus triangles of size $n$ is an $\mathbb{F}_2$-vector
space and the mapping $T\colon\mathbb{F}_2^n\rightarrow S(n)$ defined
by $\mathbf{x}\mapsto T(\mathbf{x})$ is an isomorphism.  Then, the
vector space $S(n)$ can be seen as a linear code of length $n(n+1)/2$
and dimension $n$. In general, it is difficult to find the weight
distribution of the words of a linear code; in particular, this seems to be the case for the
code $S(n)$. Here, we focus on the smallest and largest values of the
weight distribution of $S(n)$. To be precise, let
$0=w_0<w_1<w_2<\ldots<w_{m-1}<w_m$ be all the weights of the triangles
of $S(n)$. For $i\in\{0,\ldots,m\}$, an $i$-\emph{sequence} is a sequence
$\mathbf{x}$ such that $|T(\mathbf{x})|=w_i$. We denote by $W_i$ the
set of $i$-sequences.  The number of entries in a Steinhaus triangle of size $n$ is $n(n+1)/2$.
The original problem was to decide if, for some $i$, the value of $w_i$ is $n(n+1)/4$. A natural question is to ask
which are the lowest and greatest values of the $w_i$.
Obviously, only one triangle exists with
weight $0$, which is that generated by the sequence $\mathbf{0}$, so
$w_0=0$ and $W_0=\{\mathbf{0}\}$.  Our goal is to determine $w_1$,  $w_2$, $w_3$ and $w_{m}$,  and the
corresponding sets $W_1$, $W_2$, $W_3$ and $W_m$. Furthemore, we give $w_{m-1}$ and $W_{m-1}$ for sequences of length $n\equiv 1\pmod{3}$ and a conjecture for $n\equiv 0,2\pmod{3}$. 

We shall assume the length $n$ of the sequences, the size $n$ of the
triangles, and the dependence on $n$ of the values $w_i$ and the sets
$W_i$. However, if it is convenient to make them explicit, we shall use
superscripts.  Thus, $\mathbf{x}^{(n)}$ means that the sequence
$\mathbf{x}$ has length $n$, and $w_i^{(n)}$ means the value $w_i$ for
triangles of size $n$, etc. Also, we denote by
$m=m(n)$ the maximum $i$ such that there exist $i$-sequences of length
$n$.

 If we apply a clockwise rotation of 120 degrees to the graphical representation of a Steinhaus triangle, 
we obtain the Steinhaus triangle generated by the sequence
of the \mbox{right-side} of the triangle read from top to bottom (the
sequence $r(\mathbf{x})=(1,1,1,0,1,1,1)$ in  Figure~\ref{fexemple2}, left). Given $\mathbf{x}$, we define
$r(\mathbf{x})$ as the sequence which has as $j$-th entry the last entry of the $j$-th row of $T(\mathbf{x})$, that is,
\begin{equation}
\label{right}
(r(\mathbf{x}))_j=\sum_{\ell=0}^{j}{j\choose \ell}x_{n-1-j+\ell}.
\end{equation}
Applying a rotation of 240 degrees to the graphical representation of a Steinhaus triangle, 
we obtain the Steinhaus triangle generated by the sequence
of the left-side of the triangle read from bottom to top (the
sequence $\ell(\mathbf{x})=(1,0,0,1,0,0,0)$ in Figure~\ref{fexemple2}, left). Given $\mathbf{x}$, we define
the sequence $\ell(\mathbf{x})$ as the sequence which has as $j$-th entry the first
entry of the $(n-1-j)$-th row of $T(\mathbf{x})$, that is,
$$
(\ell(\mathbf{x}))_j=\sum_{\ell=0}^{n-1-j}{ n-1-j\choose \ell}x_\ell.
$$
 Applying a symmetry with respect to the height of the inferior vertex, we obtain
the triangle generated by the sequence which consists in reading $\mathbf{x}$ from right to left,
($i(\mathbf{x})=(1,0,0,1,0,0,0)$ in Figure~\ref{fexemple2}, right). Given $\mathbf{x}$, we define
$i(\mathbf{x})$ by
$$
i(\mathbf{x})=(x_{n-1},x_{n-2},\ldots,x_1,x_0).
$$
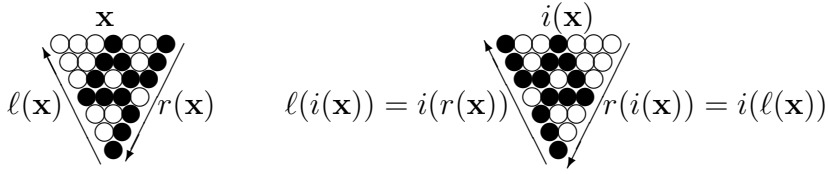
\begin{figure}[tb]
\centering
\setlength{\unitlength}{2.35mm}
\hspace{-2truecm}
\begin{picture}(7,7)
\n{0.5}{7.5} \n{1.5}{7.5} \n{2.5}{7.5} \y{3.5}{7.5} \n{4.5}{7.5} \n{5.5}{7.5} \y{6.5}{7.5}
\n{1.0}{6.5} \n{2.0}{6.5} \y{3.0}{6.5} \y{4.0}{6.5} \n{5.0}{6.5} \y{6.0}{6.5}
\n{1.5}{5.5} \y{2.5}{5.5} \n{3.5}{5.5} \y{4.5}{5.5} \y{5.5}{5.5}
\y{2.0}{4.5} \y{3.0}{4.5} \y{4.0}{4.5} \n{5.0}{4.5}
\n{2.5}{3.5} \n{3.5}{3.5} \y{4.5}{3.5}
\n{3.0}{2.5} \y{4.0}{2.5}
\y{3.5}{1.5}
\put(7.5,7.5){\vector(-1,-2){3.3}} %fletxa dreta
\put(2.8,0.7){\vector(-1,2){3.3}} %fletxa esquerra
\put(2.5,8.5){$\mathbf{x}$}
\put(6,3.5){$r(\mathbf{x}$)}
\put(-2.5,3.5){$\ell(\mathbf{x}$)}
\end{picture}
\hspace{4truecm}
\begin{picture}(7,7)
\y{0.5}{7.5} \n{1.5}{7.5} \n{2.5}{7.5} \y{3.5}{7.5} \n{4.5}{7.5} \n{5.5}{7.5} \n{6.5}{7.5}
\y{1.0}{6.5} \n{2.0}{6.5} \y{3.0}{6.5} \y{4.0}{6.5} \n{5.0}{6.5} \n{6.0}{6.5}
\y{1.5}{5.5} \y{2.5}{5.5} \n{3.5}{5.5} \y{4.5}{5.5} \n{5.5}{5.5}
\n{2.0}{4.5} \y{3.0}{4.5} \y{4.0}{4.5} \y{5.0}{4.5}
\y{2.5}{3.5} \n{3.5}{3.5} \n{4.5}{3.5}
\y{3.0}{2.5} \n{4.0}{2.5}
\y{3.5}{1.5}
\put(7.5,7.5){\vector(-1,-2){3.5}} %fletxa dreta
\put(2.8,0.7){\vector(-1,2){3.5}} %fletxa esquerra
\put(2.5,8.5){$i(\mathbf{x})$}
\put(6,3.5){$r(i(\mathbf{x}))=i(\ell(\mathbf{x}))$}
\put(-12,3.5){$\ell(i(\mathbf{x}))=i(r(\mathbf{x}))$}
\end{picture}
\caption{The sequences $r(\mathbf{x})$, $\ell(\mathbf{x})$ and $i(\mathbf{x})$ for
  $\mathbf{x}=(0,\, 0,\, 0,\, 1,\, 0,\, 0,\, 1)$.}
\label{fexemple2}
\end{figure}

The mappings
$r,\ell,i\colon \mathbb{F}_2^n\rightarrow \mathbb{F}_2^n$ are linear
and, as $r\ell=\ell r=id$ and $i^2=id$, they are bijectives. Obviously,
$\{id, r, \ell, i, ri, \ell i\}$ is a group of automorphisms of
$\mathbb{F}_2^n$ isomorphic to the dihedral  group  $D_6$. Two sequences $\mathbf{x}$
and $\mathbf{y}$ are \emph{equivalent} if there exists $f\in D_6=\{id, r, \ell, i, ri, \ell i\}$
such that $\mathbf{y}=f(\mathbf{x})$. A triangle  $T(\mathbf{x})$ is univocally determined
by $\mathbf{x}$, as well as by any of the sequences $r(\mathbf{x})$, $\ell(\mathbf{x})$ and $i(\mathbf{x})$; in particular,
triangles generated by equivalent sequences have the same weight.
Note that, given $\mathbf{x}$,  the six triangles $T(f(\mathbf{x}))$ with $f\in D_6$ are not necessarily
distinct  because of possible symmetries. For instance, in Figure~\ref{fexemple2},
we have that $r(\mathbf{x})$ and $\ell(i\mathbf{x})$ are equal, so both
generate the same triangle.

When writing sequences in form of words, we shall use a dot to represent concatenation. Thus, $101\cdot 01=10101$.
We also shall use the following notation: $\overline{x_1x_2\ldots x_p}$ stands for
the infinite sequence obtained by repeating $x_1x_2\ldots x_p$,
and $\overline{x_1x_2\ldots x_p}[n]$ is the sequence formed
by the first $n$ entries of $\overline{x_1x_2\ldots x_p}$. For instance,
$\overline{100}[4]=1001$, $\overline{100}[5]=10010$, and
$\overline{100}[6]=100100$.

\section{The cases $n\le 4$}

Proposition~\ref{casos n<=4} summarizes the sequence $w_1<\ldots <w_m$ and the sets
$W_1,\ldots, W_m$ when $n\le 4$. The proof consists of a simple checking and thus is omitted.

\begin{prop}
\label{casos n<=4}
\begin{enumerate}[\rm(i)]\setlength{\topsep}{0pt}
\item If $n=1$, then $m(1)=1$, $w_1=1$, and $W_1=\{\mathbf{1}\}$.
\item If $n=2$, then $m(2)=1$, $w_1=2$, and $W_1=\{11, \,10, \,01\}$.
\item If $n=3$, then $m(3)=2$, $w_1=3$, $w_2=4$, $W_1=\{111,\,100,\,001,\,010\}$ and $W_2=\{110, 011, 101\}$.
\item If $n=4$, then $m(4)=4$, $w_1=4$, $w_2=5$, $w_3=6$, $w_4=7$,  and
$$
\begin{array}{ll}
W_1=\{1111,\,1000,\,0001\},
& W_2 =\{0100, \,0010, \,1100, \,1010, \,0101, \,0011\},\\
W_3=\{1001, 0110, 1110, 0111\},
& W_4=\{1101,\,1011\}.
\end{array}
$$
\end{enumerate}
\end{prop}

\begin{remark}\normalfont
  Each of the sets $W_1^{(1)}$, $W_1^{(2)}$, $W_2^{(3)}$, $W_1^{(4)}$,
  $W_2^{(4)}$, and $W_4^{(4)}$ are exactly one equivalence class with respect to the action of $D_6$. The
  set $W_1^{(3)}$ is the reunion of the two equivalence classes
  $\{111,\, 100,\, 001\}$ and $\{010\}$, and $W_3^{(4)}$ is the reunion of
  two equivalence class, $\{1001,\, 1110, \,0111\}$ and $\{0110\}$.
\end{remark}

\begin{remark}\normalfont
  Observe that two triangles can have different weight, but have all
  the subtriangles of a fixed size with the same weigth.  For
  instance, in the two triangles generated by $1011\in W_4^{(4)}$ and
  $0110\in W_3^{(4)}$ (see Figure~\ref{fdw}) all subtriangle of size 2
  have weight 2,  and all subtriangles of size 3 have weight 4. Nevertheless, $|T(1011)|=7$ and $|T(0110)|=6$. 
\end{remark}

\begin{figure}[htb]
\centering
\setlength{\unitlength}{2.35mm}
\begin{picture}(19.4,3.5)
\y{0.5}{3.5} \n{1.5}{3.5} \y{2.5}{3.5} \y{3.5}{3.5}
\y{1.0}{2.5} \y{2.0}{2.5} \n{3.0}{2.5}
\n{1.5}{1.5} \y{2.5}{1.5}
\y{2.0}{0.5}
\end{picture}
\begin{picture}(7,7)
\n{0.5}{3.5} \y{1.5}{3.5} \y{2.5}{3.5} \n{3.5}{3.5}
\y{1.0}{2.5} \n{2.0}{2.5} \y{3.0}{2.5}
\y{1.5}{1.5} \y{2.5}{1.5}
\n{2.0}{0.5}
\end{picture}
\caption{Triangles generated by sequences $1011$ and $0110$.}
\label{fdw}
\end{figure}
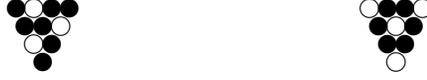

\section{The canonical basis}

In the following sections we use some bounds on the weight of the triangles
generated by the sequences of the canonical basis of
$\mathbb{F}_2^n$. Here, we obtain these bounds and, in some cases, the
exact value.

Consider the vectors of the canonical basis of $\mathbb{F}_2^n$:
$$
\begin{array}{l}
\mathbf{e}_0=1\cdot\overline{0}[n-1];\\
\mathbf{e}_k= \overline{0}[k]\cdot 1\cdot \overline{0}[n-k-1],\quad 1\le k\le n-2;\\
\mathbf{e}_{n-1}=\overline{0}[n-1]\cdot 1.
\end{array}
$$
Our first result gives the exact weight of the triangles $T(\mathbf{e}_k)$ for $k\in\{0,1,2,3\}$.

\begin{prop} Let  $n$ be the length of the considered sequences.
\label{pesos T(e0),T(e1),T(e2)}
\vspace{-\parskip}
\begin{enumerate}[\rm(i)]
\item $|T(\mathbf{e}_0)|=n$.
\item $|T(\mathbf{e}_1)|=\lfloor (3n-2)/2\rfloor$;
\item If   $n\equiv 2\pmod{4}$, then $|T(\mathbf{e}_2)|=2n-4$;\\
      otherwise,  $|T(\mathbf{e}_2)|=2n-3$.
\item if  $n\equiv 3\pmod{4}$, then   $|T(\mathbf{e}_3)|=(9n-27)/4$;\\
      otherwise, $|T(\mathbf{e}_3)|=\lfloor (9n-20)/4\rfloor$;\\
\end{enumerate}
\end{prop}
\begin{proof}
All proofs are by induction on $n$.

(i)  For $n=1$, it is trivial. Let $n\ge 2$ and assume the result true for $n-1$.
We have $\partial\mathbf{e}_0=\mathbf{e}_0^{(n-1)}$ and, by induction hypothesis, $|T(\mathbf{e}_0^{(n-1)})|=n-1$.
Then, $|T(\mathbf{e}_0)|=|\mathbf{e}_0|+|T(\mathbf{e}_0^{(n-1)})|=1+(n-1)=n$.

(ii) For $n=2$ and $n=3$ it can be checked immediately that $|T(\mathbf{e}_1^{(2)})|=2=\lfloor(3n-2)/2\rfloor$ and
$|T(\mathbf{e}_1^{(3)})|=3=\lfloor(3n-2)/2\rfloor$. Let $n\ge 4$ and assume that the result is true for values lower than $n$.
We have $\partial\mathbf{e}_1=11\cdot\overline{0}[n-3]$ and
$\partial^2\mathbf{e}_1=01\cdot\overline{0}[n-4]=\mathbf{e}_1^{(n-2)}$. Hence, by  induction hypothesis,
$$
|T(\mathbf{e}_1)|=|\mathbf{e}_1|+|\partial\mathbf{e}_1|+|T(\partial^2\mathbf{e}_1)|=1+2+|T(\mathbf{e}_1^{(n-2)})|
=3+\lfloor (3(n-2)-2)/2\rfloor=\lfloor(3n-2)/2\rfloor.
$$
\indent(iii) For  $n\in\{3,4,5,6,7\}$, it can be checked immediately that
$|T(\mathbf{e}_2^{(3)})|=3=2n-3$, $|T(\mathbf{e}_2^{(4)})|=5=2n-3$,
$|T(\mathbf{e}_2^{(5)})|=7=2n-3$,
$|T(\mathbf{e}_2^{(6)})|=8=2n-4$ and $|T(\mathbf{e}_2^{(7)})|=11=2n-3$. Assume $n\ge 8$ and that the result is true
for values lower than $n$. Observe that
$$
\begin{array}{lll}
\mathbf{e}_2^{(n)}=001\cdot\overline{0}[n-3],
& \partial\mathbf{e}_2^{(n)}=011\cdot\overline{0}[n-4],
& \partial^2\mathbf{e}_2^{(n)}=101\cdot\overline{0}[n-5], \\
\partial^3\mathbf{e}_2^{(n)}=111\cdot\overline{0}[n-6],
& \partial^4\mathbf{e}_2^{(n)}=001\cdot\overline{0}[n-7]=\mathbf{e}_2^{(n-4)}.
\end{array}
$$
Hence,
\begin{align*}
|T(\mathbf{e}_2^{(n)})|
=& |\mathbf{e}_2^{(n)}|+| \partial\mathbf{e}_2^{(n)}|+|\partial^2\mathbf{e}_2^{(n)}|
+|\partial^3\mathbf{e}_2^{(n)}|+|T(\mathbf{e}_2^{(n-4)})|\\
=& 1+2+2+3+|T(\mathbf{e}_2^{(n-4)})|\\
=& 8+|T(\mathbf{e}_2^{(n-4)})|.
\end{align*}
If $n\equiv 2\pmod{4}$, then
$|T(\mathbf{e}_2^{(n)})|=8+2(n-4)-4=2n-4$. Otherwise, $|T(\mathbf{e}_2^{(n)})|=8+2(n-4)-3=2n-3$.

(iv)  For  $n\in\{4,5,6,7\}$, it can be checked immediately that
$|T(\mathbf{e}_3^{(4)})|=4=\lfloor(9n-20)/4\rfloor$,
$|T(\mathbf{e}_3^{(5)})|=6=\lfloor(9n-20)/4\rfloor$,
$|T(\mathbf{e}_3^{(6)})|=8=\lfloor(9n-20)/4\rfloor$ and
$|T(\mathbf{e}_3^{(7)})|=9=(9n-27)/4$. Assume $n\ge 8$ and that the result is true
for values lower than $n$. Observe that
$$
\begin{array}{lll}
\mathbf{e}_3^{(n)}=0001\cdot\overline{0}[n-4],
& \partial\mathbf{e}_3^{(n)}=0011\cdot\overline{0}[n-5],
& \partial^2\mathbf{e}_3^{(n)}=0101\cdot\overline{0}[n-6], \\
\partial^3\mathbf{e}_3^{(n)}=1111\cdot\overline{0}[n-7],
& \partial^4\mathbf{e}_3^{(n)}=0001\cdot\overline{0}[n-8]=\mathbf{e}_3^{(n-4)}.
\end{array}
$$
Hence,
\begin{align*}
|T(\mathbf{e}_3^{(n)})|
=& |\mathbf{e}_3^{(n)}|+| \partial\mathbf{e}_3^{(n)}|+|\partial^2\mathbf{e}_3^{(n)}|
+|\partial^3\mathbf{e}_3^{(n)}|+|T(\mathbf{e}_3^{n-4})|\\
=& 1+2+2+4+|T(\mathbf{e}_3^{n-4})|\\
=& 9+|T(\mathbf{e}_3^{n-4})|.
\end{align*}
If $n\equiv 3\pmod{4}$, then
$|T(\mathbf{e}_3^{(n)})|=9+(9(n-4)-27)/4=(9n-27)/4$;
otherwise,
$|T(\mathbf{e}_3^{(n)})|=9+\lfloor(9(n-4)-20)/4\rfloor=\lfloor(9n-20)/4\rfloor$. \qed
\end{proof}

\begin{remark}\normalfont
  Since $\mathbf{e}_{n-1-k}=i(\mathbf{e}_k)$, we have
  $|T(\mathbf{e}_{n-1-k})|=|T(\mathbf{e}_{k})|$. Hence, to calculate
  the weight of the triangles $T(\mathbf{e}_{k})$ it is sufficient
  to consider the values $k\le (n-1)/2$. Thus, in the rest of the section we can assume $n\ge 2k+1$.
\end{remark}

Table~\ref{pesos T(ei) n petit} shows the values of
$|T(\mathbf{e}_k^{(n)})|$ for  $9\le n\le 15$ and $4\le k\le
(n-1)/2$. We see
that $|T(\mathbf{e}_k^{(n)})|\ge 2n-3$ and that, for $n=12$ and $n=14$
the inequality is strict. Proposition~\ref{pesos T(ei)} generalizes
these observations.

\begin{d2}[htb]
$$
\begin{array}{|r|l|l|ll|ll|}
\hline
                   n   & 9  & 10 & 11 &    & 12 &    \\
\hline
                   k   & 4  & 4  & 4  & 5  & 4  & 5 \\
\hline
|T(\mathbf{e}_k^{(n)})| & 17 & 19 & 21 & 21 & 22 & 23  \\
\hline
\end{array}
$$
$$
%\hspace{-0.5truecm}
\begin{array}{|r|lll|lll|llll|}
\hline
n                       & 13 &    &    & 14 &    &    & 15 &    &    &     \\
\hline
k                       & 4  & 5  & 6  & 4  &  5 &  6 &  4 & 5  & 6  & 7  \\
\hline
|T(\mathbf{e}_k^{(n)})| & 27 & 24 & 25 & 30 & 30 & 26 & 33 & 33 & 33 & 27 \\
\hline
\end{array}
$$
\caption{Weights of $T(\mathbf{e}_k^{(n)})$ for $9\le n\le 15$ and $4\le k\le(n-1)/2$. }
\label{pesos T(ei) n petit}
\end{d2}

\begin{prop}
\label{pesos T(ei)}
Let $k\ge 4$ and $n\ge 9$ be integers  with $k\le (n-1)/2$. Then,
$|T(\mathbf{e}_k^{(n)})|\ge 2n-3$. Moreover, if $n$ is even, the inequality is strict:
$|T(\mathbf{e}_k^{(n)})|>2n-3$.
\end{prop}
\begin{proof}
Table~\ref{pesos T(ei) n petit} shows that the result is true for $9\le n\le 15$.
Let $n\ge 16$ and assume the result true for values lower than $n$ (and greater or equal than 9).
We use the facts that $r(\mathbf{e}_k^{(n-1)})=\partial r(\mathbf{e}_k^{(n)})$ and
$|T(\mathbf{e}_k^{(n)})|=|T(\mathbf{e}_k^{(n-1)})|+|r(\mathbf{e}_k^{(n)})|$. The argument is casuistic on the weight of
$r(\mathbf{e}_k^{(n)})$.

If $|r(\mathbf{e}_k^{(n)})|\ge 4$, by induction hypothesis we have,
$$
|T(\mathbf{e}_k^{(n)})|=|T(\mathbf{e}_k^{(n-1)})|+|r(\mathbf{e}_k^{(n)})|\ge 2(n-1)-3+4=2n-1>2n-3.
$$
\indent If $|r(\mathbf{e}_k^{(n)})|=3$, then $|r(\mathbf{e}_k^{(n-1)})|=|\partial r(\mathbf{e}_k^{(n)})|\ge 1$ and
$$
|T(\mathbf{e}_k^{(n)})|=|T(\mathbf{e}_k^{(n-2)})|+|r(\mathbf{e}_k^{(n-1)})|+|r(\mathbf{e}_k^{(n)})|
\ge 2(n-2)-3+1+3=2n-3.
$$
Moreover, if $n$ is even, then $n-2$ is also even and by induction hypothesis the inequality is  strict.

Now consider the case when $|r(\mathbf{e}_k^{(n)})|=2$. First, suppose that the two ones are the last two coordinates.
If $r(\mathbf{e}_k^{(n)})=\mathbf{u}=\overline{0}[n-2]\cdot 11$, then $\mathbf{e}_k^{(n)}=\ell(\mathbf{u})$.
Now, depending on the parity of $n$, we have
$\ell(\mathbf{u})=01\cdot \overline{0}[n-2]$ or $\ell(\mathbf{u})=11\cdot \overline{0}[n-2]$, both distinct
from  $\mathbf{e}_k^{(n)}$, a contradiction. Hence, the two ones of $r(\mathbf{e}_k^{(n)})$ are not the last two coordinates. This implies $|r(\mathbf{e}_k^{(n-1)})|=|\partial r(\mathbf{e}_k^{(n)})|\ge 2$. Therefore,
$$
|T(\mathbf{e}_k^{(n)})|=|T(\mathbf{e}_k^{(n-2)})|+|r(\mathbf{e}_k^{(n-1)})|+|r(\mathbf{e}_k^{(n)})|
> 2(n-2)-3+2+2=2n-3.
$$
As before, if $n$ is even, $n-2$ is also even and the inequality is strict.

Finally, consider the case $|r(\mathbf{e}_k^{(n)})|=1$. The coordinate $n-k-1$ of $r(\mathbf{e}_k^{(n)})$ is $1$,
so $r(\mathbf{e}_k^{(n)})=\mathbf{e}_{n-k-1}=\overline{0}[n-k-1]\cdot 1\cdot\overline{0}[k]$. The first three derivatives are
\begin{align*}
r(\mathbf{e}_k^{(n-1)})=\partial r(\mathbf{e}_k^{(n)})
&=\overline{0}[n-k-2]\cdot 11\cdot\overline{0}[k-1],\\
r(\mathbf{e}_k^{(n-2)})=\partial r(\mathbf{e}_k^{(n-1)})
&=\overline{0}[n-k-3]\cdot 101\cdot\overline{0}[k-2],\\
r(\mathbf{e}_k^{(n-3)})=\partial r(\mathbf{e}_k^{(n-2)})
&=\overline{0}[n-k-4]\cdot 1111\cdot\overline{0}[k-3].
\end{align*}
Therefore,
\begin{align*}
|T(\mathbf{e}_k^{(n)})|
&=|T(\mathbf{e}_k^{(n-4)})|
+|r(\mathbf{e}_k^{(n-3)})|+|r(\mathbf{e}_k^{(n-2)})|+|r(\mathbf{e}_k^{(n-1)})|+|r(\mathbf{e}_k^{(n)})|\\
&\ge 2(n-4)-3+4+2+2+1\\
&=2n-2>2n-3.\qed
\end{align*}
\end{proof}

\section{$1$-sequences}

In this section we assume $n\ge 4$. The unique sequence of
$W_0^{(n-1)}$, which is $\mathbf{0^{(n-1)}}=\overline{0}[n-1]$, has
the two primitives $\mathbf{0}\in W_0$ and
$\mathbf{a}_1=\mathbf{1}$. We shall see that $W_1$ is exactly the
equivalence class of $\mathbf{a}_1$. As
$\mathbf{a}_1=i(\mathbf{a}_1)$, the class of $\mathbf{a}_1$ contains
the three sequences
$$
\mathbf{a}_1=\mathbf{1}=\overline{1}[n], \quad
\mathbf{a}_2=r(\mathbf{a}_1)=1\cdot\overline{0}[n-1],\quad
\mathbf{a}_3=\ell(\mathbf{a}_1)=\overline{0}[n-1]\cdot 1.
$$

\begin{remark}
\normalfont
Note that $\mathbf{a}_2=\mathbf{e}_0$ i $\mathbf{a}_3=\mathbf{e}_{n-1}$ (see Figure~\ref{fai}).
\end{remark}

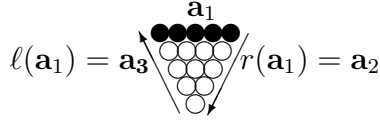
\begin{figure}[htb]
\centering
\setlength{\unitlength}{2.35mm}
\begin{picture}(18,6)
\y{8.5}{4.5}  \y{9.5}{4.5} \y{10.5}{4.5} \y{11.5}{4.5} \y{12.5}{4.5}
\n{9.0}{3.5}  \n{10.0}{3.5} \n{11.0}{3.5} \n{12.0}{3.5}
\n{9.5}{2.5}  \n{10.5}{2.5} \n{11.5}{2.5}
\n{10.0}{1.5} \n{11.0}{1.5}
\n{10.5}{0.5}
\put(13.5,4.5){\vector(-1,-2){2.3}} %fletxa dreta
\put(9.6,0){\vector(-1,2){2.3}} %fletxa esquerra
\put(10,5.5){$\mathbf{a}_1$}
\put(13,2.5){$r(\mathbf{a}_1)=\mathbf{a}_2$}
\put(0,2.5){$\ell(\mathbf{a}_1)=\mathbf{a_3}$}
\end{picture}
\caption{The sequences $\mathbf{a}_1$, $\mathbf{a}_2=r(\mathbf{a}_1)$ i $\mathbf{a}_3=\ell(\mathbf{a}_1)$ for $n=5$.}
\label{fai}
\end{figure}

\begin{remark}
\label{pesos T(ai)}
\normalfont
It is clear that $|\mathbf{a}_1|=n$ and $|\mathbf{a}_2|=|\mathbf{a}_3|=1$. Moreover,
$|T(\mathbf{a}_1)|=|T(\mathbf{a}_2)|=|T(\mathbf{a}_3)|=n$.
\end{remark}

We define $A=A^{(n)}=\{\mathbf{a}_1,\,\mathbf{a}_2,\,\mathbf{a}_3\}$.
H.~Harborth~\cite{Harborth} observed that the minimum weight of non-zero Steinhaus triangles of size $n$ is $n$
and that $|T(\mathbf{a}_1)|=|T(\mathbf{a}_2)|=|T(\mathbf{a}_3)|=n$; that is,  $w_1=n$ and $A\subseteq W_1$.
For the sake of completeness, in this section we include the  proof that $w_1=n$; we also prove that $W_1=A$.

\begin{prop}
\label{auxW1}
For $n\ge 4$, if $\mathbf{x}\in\mathbb{F}_2^n\setminus(W_0\cup A)$, then $|T(\mathbf{x})|>n$.
\end{prop}
\begin{proof}
By induction on $n$. For $n=4$, according to Proposition~\ref{casos n<=4}, we have $W_1=A$ and the result holds.
Let $n\ge 5$ and assume that the result is true for $n-1$.

Let $\mathbf{x}\in\mathbb{F}_2^n\setminus(W_0\cup A)$.
If $\partial\mathbf{x}\not\in W_0^{(n-1)}\cup A^{(n-1)}$, then we can apply the induction hypothesis and we have
$|T(\partial\mathbf{x})|>n-1$. Since $\mathbf{x}\ne\mathbf{0}$, we have $|\mathbf{x}|\ge 1$. Then,
$$
|T(\mathbf{x})|=|\mathbf{x}|+|T(\partial\mathbf{x})|> 1+(n-1)=n,
$$
as we wanted to prove. Now assume that $\partial\mathbf{x}\in
W_0^{(n-1)}\cup A^{(n-1)}$.  If $\partial\mathbf{x}\in
W_0^{(n-1)}=\mathbf{0}^{(n-1)}$, then
$\mathbf{x}\in\{\mathbf{0},\mathbf{1}\}\subset W_0\cup A$, a
contradiction. Hence, $\partial\mathbf{x}\in A^{n-1}$, that is,
$\partial\mathbf{x}=\mathbf{a}_i^{(n-1)}$ for some $i\in[3]$. Since
$|T(\partial\mathbf{x})|=|T(\mathbf{a}_i^{(n-1)})|=n-1$, we have
$|T(\mathbf{x})|=|\mathbf{x}|+|T(\partial\mathbf{x})|=|\mathbf{x}|+n-1$. If
we check that $|\mathbf{x}|\ge 2$, then $|T(\mathbf{x})|>n$ as wanted.
Table~\ref{d1x=ai} gives, for each
$\partial\mathbf{x}=\mathbf{a}_i^{(n-1)}$, the two primitives
$\mathbf{x}$ of $\mathbf{a}_i^{(n-1)}$. In two cases we have
$\mathbf{x}\in A$, against the hypothesis. In the other four cases,
bearing in mind that $n\ge 5$, it is immediately obtained that
$|\mathbf{x}|\ge 2$.\qed

\begin{d2}[htb]
$$
\begin{array}{|l|l|}
\hline
\partial\mathbf{x} & \mathbf{x} \\
\hline
\mathbf{a}_1^{(n-1)}=\overline{1}[n-1] & \overline{01}[n] \\
                                       & \overline{10}[n] \\
\hline
\mathbf{a}_2^{(n-1)}=1\cdot\overline{0}[n-1] & 0\cdot\overline{1}[n-1] \\
                                             & 1\cdot\overline{0}[n-1]=\mathbf{a}_2 \\
\hline
\mathbf{a}_3^{(n-1)}=\overline{0}[n-1]\cdot 1 & \overline{0}[n-1]\cdot 1=\mathbf{a}_3 \\
                                              & \overline{1}[n-1]\cdot 0 \\
\hline
\end{array}
$$
\caption{Sequences $\mathbf{x}$ with $\partial\mathbf{x}\in A^{(n-1)}$ ($n\ge 5$).}
\label{d1x=ai}
\end{d2}

\end{proof}

As a consequence of Proposition~\ref{auxW1}, we have the following theorem:

\begin{thm}
\label{thmW1} For $n\ge 4$, we have $w_1=n$ i $W_1=\{\mathbf{a}_1,\,\mathbf{a}_2,\,\mathbf{a}_3\}$.
\end{thm}
\begin{proof}
The triangles generated by $\mathbf{a}_1$, $\mathbf{a}_2$ and  $\mathbf{a}_3$ have weight $n$
and, by Proposition~\ref{auxW1}, all the triangles  generated by sequences of $\mathbb{F}_2^n\setminus(W_0\cup A)$
have weight $>n$. Hence, $W_1=A$ i $w_1=n$. \qed
\end{proof}

\section{$2$-sequences}

For $n=1$ and $n=2$, we have $m(n)=1$, so to study $2$-sequences it must be $n\ge 3$. The case $n=3$
has been studied in Proposition~\ref{casos n<=4}, and it turns out that $w_2^{(3)}=4$ and
$W_2^{(3)}=\{110,\, 011,\, 101\}$. In this section we assume $n\ge 4$.

To find $w_2$ and $W_2$ we use an inductive argument on the size
$n$ of the triangles, but the cases $n=6$ and $n=7$ do not follow the general rule.
Thus, the inductive argument needs to begin with the basic case  $n=8$.
Then, for $4\le n\le 8$, we use a computer-aided exhaustive search which gives
the results displayed in Table~\ref{W2 casos 4<=n<=8}. For $n\in\{4,5,8\}$, the set
$W_2$ is an equivalence class. The set $W_2^{(6)}$ is the reunion of the two equivalence classes
$\{010000,\,110000,\,000010,\,101010,\,010101,\,000011\}$ and $\{001000,\,000100,001100\}$, 
and $W_2^{(7)}$ is the reunion of the two equivalence classes $\{0100000,\,0000010,\,0101010\}$ and $\{0001000\}$.
Note that in Table~\ref{W2 casos 4<=n<=8} the
value $w_2$ depends on $n$ according to the formula 
$w_2=\lfloor 3n/2\rfloor-1=\lfloor (3n-2)/2\rfloor$ (for $n=3$, the formula gives $3$ instead of the correct one $w_2=4$).
 
\begin{d2}[htb]
$$
\begin{array}{l|l|l}
n & w_2 & W_2 \\
\hline
4 & 5 &\{0100,\,0011,\,1010,\,0010,\,0101,\,1100\}\\
5 & 6 & \{01000,\,00010, \,01010\}\\
6 & 8 & \{010000, \,110000, \,001000, \,000100, \,001100, \,000010, \,101010, \,010101, \,000011\}\\
7 & 9 & \{0100000, \,0001000, \,0000010, \,0101010\}\\
8 & 11 & \{01000000, \,11000000, \,00000010, \,10101010, \,01010101, \,00000011\}\\
\end{array}
$$
\vspace{-0.4truecm}
\caption{The sets $W_2$ and the values $w_2$ for $4\le n\le 8$.}
\label{W2 casos 4<=n<=8}
\end{d2}

\begin{d2}
$$
\begin{array}{l|l|l}
n & \lfloor 3n/2\rfloor & \mbox{sequences $\mathbf{x}$ with $|T(\mathbf{x})|=\lfloor 3n/2\rfloor $} \\
\hline
4  & 6  & \{1001,\,0110,\,1110,\,0111\}                          \\
5  & 7  & \{11000, \,00100, \,01100, \,00110, \,10101, \,00011\} \\
6  & 9  & ---                                                    \\
7  & 10 &  \{1010101,\,1100000, \,0000011\}                      \\
8  & 12 & ---                                                    \\
\end{array}
$$
\caption{Sequences $\mathbf{x}$ with $|T(\mathbf{x})|=\lfloor 3n/2\rfloor $, for $4\le n\le 8$.}
\label{W3 casos 4<=n<=8}
\end{d2}

Table~\ref{W3 casos 4<=n<=8} gives the sets of sequences which
generate triangles of weight $\lfloor 3n/2\rfloor$.  We see that for
$n\in\{4,5,7\}$, we have $w_3=\lfloor 3n/2\rfloor=w_2+1$ and, for
$n\in\{6,8\}$, $w_3>\lfloor 3n/2\rfloor$. This will be useful later.

The primitives of $\mathbf{a}_1^{(n-1)}$ are $\mathbf{b}_1=\overline{10}[n]$ and $\mathbf{b}_4=\overline{01}[n]$.
We shall see that $W_2$ consists exactly of $\mathbf{b}_1$, $\mathbf{b}_4$ and its equivalent. Define
$$
\begin{array}{lll}
\mathbf{b}_1=\overline{10}[n],
& \mathbf{b}_2=01\cdot\overline{0}[n-2],
& \mathbf{b}_3=\overline{0}[n-2]\cdot 11, \\
\mathbf{b}_4=\overline{01}[n],
&\mathbf{b}_5=11\cdot \overline{0}[n-2],
& \mathbf{b}_6=\overline{0}[n-2]\cdot 10.
\end{array}
$$

\begin{remark}
\normalfont
Note that, if $n$ is even, then
$\mathbf{b}_2=r(\mathbf{b}_1)$,  $\mathbf{b}_3=\ell(\mathbf{b}_1)$,  $\mathbf{b}_4=i(\mathbf{b}_1)$,
$\mathbf{b}_5=r(\mathbf{b}_4)$ and $\mathbf{b}_6=\ell(\mathbf{b}_4)$
(see top of Figure~\ref{fbi}). If $n$ is odd, then
 $\mathbf{b}_5=r(\mathbf{b}_1)$,  $\mathbf{b}_3=\ell(\mathbf{b}_1)$, $\mathbf{b}_2=r(\mathbf{b}_4)$
and  $\mathbf{b}_6=\ell(\mathbf{b}_4)$ (see bottom of Figure~\ref{fbi}).
\end{remark}

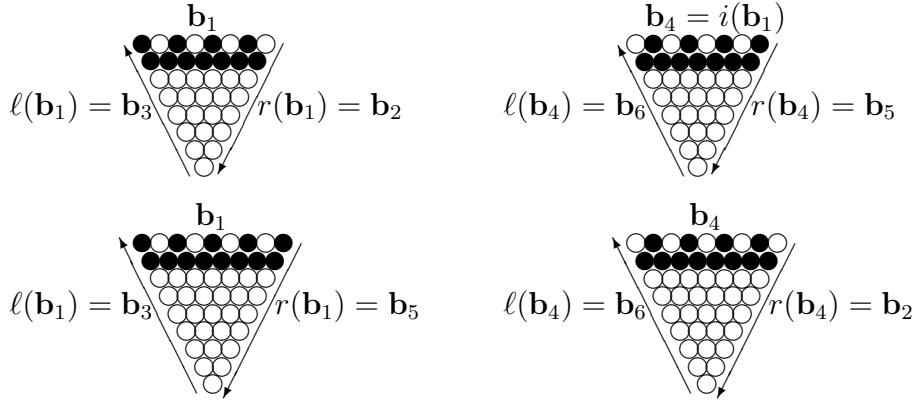
\begin{figure}[htb]
\centering
\setlength{\unitlength}{2.35mm}
\begin{picture}(14,9)
\y{0.5}{7.5} \n{1.5}{7.5} \y{2.5}{7.5} \n{3.5}{7.5} \y{4.5}{7.5} \n{5.5}{7.5} \y{6.5}{7.5} \n{7.5}{7.5}
\y{1.0}{6.5} \y{2.0}{6.5} \y{3.0}{6.5} \y{4.0}{6.5} \y{5.0}{6.5} \y{6.0}{6.5} \y{7.0}{6.5}
\n{1.5}{5.5} \n{2.5}{5.5} \n{3.5}{5.5} \n{4.5}{5.5} \n{5.5}{5.5} \n{6.5}{5.5}
\n{2.0}{4.5} \n{3.0}{4.5} \n{4.0}{4.5} \n{5.0}{4.5} \n{6.0}{4.5}
\n{2.5}{3.5} \n{3.5}{3.5} \n{4.5}{3.5} \n{5.5}{3.5}
\n{3.0}{2.5} \n{4.0}{2.5} \n{5.0}{2.5}
\n{3.5}{1.5} \n{4.5}{1.5}
\n{4.0}{0.5}
\put(8.5,7.5){\vector(-1,-2){3.7}} %fletxa dreta
\put(3.2,0){\vector(-1,2){3.7}} %fletxa esquerra
\put(3,8.5){$\mathbf{b}_1$}
\put(7,3.5){$r(\mathbf{b}_1)=\mathbf{b}_2$}
\put(-7,3.5){$\ell(\mathbf{b}_1)=\mathbf{b}_3$}
\end{picture}
\hspace{3truecm}
\begin{picture}(14,10)
\n{0.5}{7.5} \y{1.5}{7.5} \n{2.5}{7.5} \y{3.5}{7.5} \n{4.5}{7.5} \y{5.5}{7.5} \n{6.5}{7.5} \y{7.5}{7.5}
\y{1.0}{6.5} \y{2.0}{6.5} \y{3.0}{6.5} \y{4.0}{6.5} \y{5.0}{6.5} \y{6.0}{6.5} \y{7.0}{6.5}
\n{1.5}{5.5} \n{2.5}{5.5} \n{3.5}{5.5} \n{4.5}{5.5} \n{5.5}{5.5} \n{6.5}{5.5}
\n{2.0}{4.5} \n{3.0}{4.5} \n{4.0}{4.5} \n{5.0}{4.5} \n{6.0}{4.5}
\n{2.5}{3.5} \n{3.5}{3.5} \n{4.5}{3.5} \n{5.5}{3.5}
\n{3.0}{2.5} \n{4.0}{2.5} \n{5.0}{2.5}
\n{3.5}{1.5} \n{4.5}{1.5}
\n{4.0}{0.5}
\put(8.5,7.5){\vector(-1,-2){3.7}} %fletxa dreta
\put(3.2,0){\vector(-1,2){3.7}} %fletxa esquerra
\put(1,8.5){$\mathbf{b}_4=i(\mathbf{b}_1)$}
\put(7,3.5){$r(\mathbf{b}_4)=\mathbf{b}_5$}
\put(-7,3.5){$\ell(\mathbf{b}_4)=\mathbf{b}_6$}
\end{picture}
\\[5mm]
\begin{picture}(14,9)
\y{0.5}{7.5} \n{1.5}{7.5} \y{2.5}{7.5} \n{3.5}{7.5} \y{4.5}{7.5} \n{5.5}{7.5} \y{6.5}{7.5} \n{7.5}{7.5} \y{8.5}{7.5}
\y{1.0}{6.5} \y{2.0}{6.5} \y{3.0}{6.5} \y{4.0}{6.5} \y{5.0}{6.5} \y{6.0}{6.5} \y{7.0}{6.5} \y{8.0}{6.5}
\n{1.5}{5.5} \n{2.5}{5.5} \n{3.5}{5.5} \n{4.5}{5.5} \n{5.5}{5.5} \n{6.5}{5.5} \n{7.5}{5.5}
\n{2.0}{4.5} \n{3.0}{4.5} \n{4.0}{4.5} \n{5.0}{4.5} \n{6.0}{4.5} \n{7.0}{4.5}
\n{2.5}{3.5} \n{3.5}{3.5} \n{4.5}{3.5} \n{5.5}{3.5} \n{6.5}{3.5}
\n{3.0}{2.5} \n{4.0}{2.5} \n{5.0}{2.5} \n{6.0}{2.5}
\n{3.5}{1.5} \n{4.5}{1.5} \n{5.5}{1.5}
\n{4.0}{0.5} \n{5.0}{0.5}
\n{4.5}{-0.5}
\put(9.5,7.5){\vector(-1,-2){4.4}} %fletxa dreta
\put(3.6,-1){\vector(-1,2){4.4}} %fletxa esquerra
\put(3.5,8.5){$\mathbf{b}_1$}
\put(8,3.5){$r(\mathbf{b}_1)=\mathbf{b}_5$}
\put(-7,3.5){$\ell(\mathbf{b}_1)=\mathbf{b}_3$}
\end{picture}
\hspace{3truecm}
\begin{picture}(14,9)
\n{0.5}{7.5} \y{1.5}{7.5} \n{2.5}{7.5} \y{3.5}{7.5} \n{4.5}{7.5} \y{5.5}{7.5} \n{6.5}{7.5} \y{7.5}{7.5} \n{8.5}{7.5}
\y{1.0}{6.5} \y{2.0}{6.5} \y{3.0}{6.5} \y{4.0}{6.5} \y{5.0}{6.5} \y{6.0}{6.5} \y{7.0}{6.5} \y{8.0}{6.5}
\n{1.5}{5.5} \n{2.5}{5.5} \n{3.5}{5.5} \n{4.5}{5.5} \n{5.5}{5.5} \n{6.5}{5.5} \n{7.5}{5.5}
\n{2.0}{4.5} \n{3.0}{4.5} \n{4.0}{4.5} \n{5.0}{4.5} \n{6.0}{4.5} \n{7.0}{4.5}
\n{2.5}{3.5} \n{3.5}{3.5} \n{4.5}{3.5} \n{5.5}{3.5} \n{6.5}{3.5}
\n{3.0}{2.5} \n{4.0}{2.5} \n{5.0}{2.5} \n{6.0}{2.5}
\n{3.5}{1.5} \n{4.5}{1.5} \n{5.5}{1.5}
\n{4.0}{0.5} \n{5.0}{0.5}
\n{4.5}{-0.5}
\put(9.5,7.5){\vector(-1,-2){4.4}} %fletxa dreta
\put(3.6,-1){\vector(-1,2){4.4}} %fletxa esquerra
\put(3.5,8.5){$\mathbf{b}_4$}
\put(8,3.5){$r(\mathbf{b}_4)=\mathbf{b}_2$}
\put(-7,3.5){$\ell(\mathbf{b}_4)=\mathbf{b}_6$}
\end{picture}
\caption{The sequences $\mathbf{b}_i$ for $n=8$ in the top (case $n$ even), and $n=9$ in the bottom (case $n$ odd).}
\label{fbi}
\end{figure}

\begin{remark}
\label{pesos bi}
\normalfont
We have $|\mathbf{b}_2|=|\mathbf{b}_6|=1$, $|\mathbf{b}_3|=|\mathbf{b}_5|=2$, 
$|\mathbf{b}_1|=\lceil n/2\rceil|$ and  $|\mathbf{b}_4|=\lfloor n/2\rfloor$.
\end{remark}

\begin{remark}
\normalfont
According to Table~\ref{W2 casos 4<=n<=8}, for $n=4$, we have $W_2^{(4)}=\{\mathbf{b}_i^{(4)}: i\in[6]\}$.
Also, for $n=8$,  $W_2^{(8)}=\{\mathbf{b}_i^{(8)}: i\in[6]\}$.
\end{remark}

Define $B=B^{(n)}=\{\mathbf{b}_i: i\in[6]\}$.

\begin{prop} Let $n\ge 4$ be an integer.
\label{pesos T(bi)}
\vspace{-\parskip}
\begin{enumerate}[\rm (i)]
\item If $n$ is even, then $|T(\mathbf{b}_i)|=(3n-2)/2$ for all $i\in[6]$.
\item If $n$ is odd, then \\[1mm]
$|T(\mathbf{b}_1)|=|T(\mathbf{b}_3)|=|T(\mathbf{b}_5)|= (3n-1)/2$, and\\
$|T(\mathbf{b}_2)|=|T(\mathbf{b}_4)|=|T(\mathbf{b}_6)|= (3n-3)/2$.
\end{enumerate}
\end{prop}
\begin{proof}
We use that  $\partial\mathbf{b}_1=\partial\mathbf{b}_4=\mathbf{1}^{(n-1)}=\mathbf{a}_1^{(n-1)}$
and $|\mathbf{a}_1^{(n-1)}|=n-1$.

(i) Since $n$ is even, we have $|\mathbf{b}_1|=n/2$. Then,
$$
|T(\mathbf{b}_1)|=|\mathbf{b}_1|+|T(\mathbf{a}_1^{(n-1)})|=n/2+n-1=(3n-2)/2.
$$
The sequences $\mathbf{b}_2$, $\mathbf{b}_3$, $\mathbf{b}_4$, $\mathbf{b}_5$, and
$\mathbf{b}_6$ are equivalent to $\mathbf{b}_1$. Hence, $|T(\mathbf{b}_i)|=|T(\mathbf{b}_1)|=(3n-2)/2$
for all $i\in[6]$.

(ii) Since $n$ is odd, we have $|\mathbf{b}_1|=(n+1)/2$ and $|\mathbf{b}_4|=(n-1)/2$. Then,
\begin{align*}
|T(\mathbf{b}_1)|&=|\mathbf{b}_1|+|T(\mathbf{a}_1^{(n-1)})|=(n+1)/2+n-1=(3n-1)/2, \mbox{ and}\\
|T(\mathbf{b}_4)|&=|\mathbf{b}_4|+|T(\mathbf{a}_1^{(n-1)})|=(n-1)/2+n-1=(3n-3)/2.
\end{align*}
The sequences $\mathbf{b}_3$ and $\mathbf{b}_5$ are equivalent to $\mathbf{b}_1$. Hence,
$|T(\mathbf{b}_3)|=|T(\mathbf{b}_5)|=|T(\mathbf{b}_1)|=(3n-1)/2$.
The sequences  $\mathbf{b}_2$ and  $\mathbf{b}_6$ are equivalent to $\mathbf{b}_4$. Hence,
$|T(\mathbf{b}_2)|=|T(\mathbf{b}_6)|=|T(\mathbf{b}_4)|=(3n-3)/2$.\qed
\end{proof}

\begin{prop}
\label{auxW2}
Let $n\ge 8$ an integer and $\mathbf{x}\in\mathbb{F}_2^n\setminus(W_0\cup W_1\cup B)$. Then,
$|T(\mathbf{x})|>\lfloor(3n-1)/2\rfloor$;
\end{prop}
\begin{proof}
  By induction on $n$. For $n=8$ we have seen in
  Table~\ref{W2 casos 4<=n<=8} that the sequences in $B^{(8)}$ are
  exactly the $2$-sequences and that they generate triangles of
  weight $11=\lfloor(3n-1)/2\rfloor$.  Therefore, all sequences in
  $\mathbb{F}_2^8\setminus(W_0\cup W_1\cup B)$ generate triangles
  of weight $>11=\lfloor(3n-1)/2\rfloor$. Thus, for $n=8$ the result holds.  Now,
  suppose that $n\ge 9$ and that the result holds for $n-1$.  Let
  $\mathbf{x}\in\mathbb{F}_2^n\setminus(W_0\cup W_1\cup B)$.  Since
  $\mathbf{x}\not\in W_0$, it follows $|\mathbf{x}|\ge 1$.

  If $|\mathbf{x}|=1$, then $\mathbf{x}=\mathbf{e}_k$ for some
  $k\in\{0,\ldots,n-1\}$. By symmetry we can consider $k\le (n-1)/2$.
  As $\mathbf{x}\ne \mathbf{a_2}=\mathbf{e}_0$
  and $\mathbf{x}\ne \mathbf{b_2}=\mathbf{e}_1$, it must be $k\ge 2$.
  By Proposition~\ref{pesos T(e0),T(e1),T(e2)}, if $k=2$ we have $|T(\mathbf{x})|=
 |T(\mathbf{e}_2)|\ge 2n-4>\lfloor(3n-1)/2\rfloor$; for  $k=3$, we have
$|T(\mathbf{x})|=
 |T(\mathbf{e}_3)|\ge (9n-27)/4\ge \lfloor(3n-1)/2\rfloor$; and for $k\ge 4$,
 we have  $|T(\mathbf{x})|=|T(\mathbf{e}_k)|\ge 2n-3>\lfloor(3n-1)/2\rfloor$, as we wanted.  Then,
  in the following we can assume $|\mathbf{x}|\ge 2$.

Consider the sequence $\partial\mathbf{x}$. If
$\partial\mathbf{x}\notin W_0^{(n-1)}\cup W_1^{(n-1)}\cup B^{(n-1)}$,
by induction hypothesis we have $|T(\partial\mathbf{x})|>\lfloor (3(n-1)-1)/2\rfloor=\lfloor (3n-4)/2\rfloor$.
Then,
$$
|T(\mathbf{x})|=|\mathbf{x}|+|T(\partial\mathbf{x})|
> 2+\lfloor (3n-4)/2\rfloor 
= \lfloor 3n/2\rfloor \ge\lfloor(3n-1)/2\rfloor.
$$
\indent Now, suppose $\partial\mathbf{x}\in W_0\cup W_1\cup B$. If $\partial\mathbf{x}\in W_0$, then
$\mathbf{x}\in\{\mathbf{0},\mathbf{1}\}\subset W_0\cup W_1$, a contradiction.  Suppose
 $\partial\mathbf{x}=\mathbf{a}_i^{(n-1)}\in W_1$ for some $i\in[3]$.
Table~\ref{d1x=ai2} gives the six  possible $\mathbf{x}\in\mathbb{F}_2^n$ such that
$\partial\mathbf{x}\in W_1$. The two possible $\mathbf{x}$ which do not belong to $W_1\cup W_2$ have weight $n-1$. For these $\mathbf{x}$, we have $|T(\mathbf{x})|=|\mathbf{x}|+|T(\mathbf{a}_i^{(n-1)})|=n-1+n-1=2n-2>\lfloor(3n-1)/2\rfloor$.

\begin{d2}[htb]
$$
\begin{array}{|l|l|}
\hline
\partial\mathbf{x} & \mathbf{x} \\
\hline
\mathbf{a}_1^{(n-1)}=\overline{1}[n-1] & \overline{10}[n]=\mathbf{b}_1 \\
                                       & \overline{01}[n]=\mathbf{b}_4 \\
\hline
\mathbf{a}_2^{(n-1)}=1\cdot\overline{0}[n-1] & 0\cdot\overline{1}[n-1] \\
                                             & 1\cdot\overline{0}[n-1]=\mathbf{a}_2 \\
\hline
\mathbf{a}_3^{(n-1)}=\overline{0}[n-1]\cdot 1 & \overline{0}[n-1]\cdot 1=\mathbf{a}_3 \\
                                              & \overline{1}[n-1]\cdot 0 \\
\hline
\end{array}
$$
\vspace{-3truemm}
\caption{Sequences $\mathbf{x}$ with $\partial\mathbf{x}\in W_1^{(n-1)}$.}
\label{d1x=ai2}
\end{d2}

Finally, consider the case when $\partial\mathbf{x}=\mathbf{b}_i\in B^{(n-1)}$.
By Proposition~\ref{pesos T(bi)}, $|T(\mathbf{b}_i^{(n-1)})|\ge (3n-6)/2$. 
Table~\ref{pes x+dxb} shows the  sequences $\mathbf{x}$ such that $\partial\mathbf{x}\in B^{(n-1)}$.
In four cases, $\mathbf{x}\in B$, a contradiction. In the eight remaining cases, then $|\mathbf{x}|>3$ holds
(recall $n\ge 9$). Then,
$$
|T(\mathbf{x})|=|\mathbf{x}|+|T(\mathbf{b}_i)|> 3+(3n-6)/2=3n/2\ge\lfloor (3n-1)/2\rfloor. \qed
$$
\end{proof}

As a conclusion of the section, we have the following theorem:

\begin{thm}
\label{W2}
Let $n\ge 6$. Then $w_2=\lfloor (3n-2)/2\rfloor$ and the set $W_2$ is the following:
\vspace{-\parskip}
\begin{enumerate}[\rm (i)]
\item if $n=6$, then $W_2=\{\mathbf{b}_1,\,\mathbf{b}_2,\,\mathbf{b}_3,\,\mathbf{b}_4,\,\mathbf{b}_5, \,\mathbf{b}_6,\,
001000,\,000100, \,001100\}$;
\item if $n=7$, then  $W_2=\{\mathbf{b}_2, \,\mathbf{b}_4,\,\mathbf{b}_6,\,0001000\}$;
\item if $n\ge 8$ and $n$ is even, then  $W_2=\{\mathbf{b}_1, \,\mathbf{b}_2,\,\mathbf{b}_3, \,\mathbf{b}_4,\,
\mathbf{b}_5,\,\mathbf{b}_6\}$;
\item if $n\ge 9$ and $n$ is odd, then  $W_2=\{\mathbf{b}_2,\, \mathbf{b}_4,\,\mathbf{b}_6\}$.
\end{enumerate}
\end{thm}
\begin{proof}
Note that if $n$ is even, then $\lfloor (3n-2)/2\rfloor=(3n-2)/2$; if $n$ is odd, $\lfloor (3n-2)/2\rfloor=(3n-3)/2$.

(i) i (ii) The cases $n=6$ and $n=7$ follow from Table~\ref{W2 casos 4<=n<=8}.

(iii) If $n\ge 8$ is even, the sequences in $B$ generate triangles of
weight $(3n-2)/2$, and the sequences in $\mathbb{F}_2^n\setminus(W_0\cup W_1\cup B)$ generate
triangles of weight strictly greater than $(3n-2)/2$. Therefore, $w_2=(3n-2)/2$ and $W_2=B$.

(iv) If $n\ge 9$ is odd, the sequences  $\mathbf{b}_2$, $\mathbf{b}_4$, and
$\mathbf{b}_6$ generate triangles of weight $(3n-3)/2$, the sequences
$\mathbf{b}_1$, $\mathbf{b}_3$, i $\mathbf{b}_5$ generate
triangles of weight $(3n-1)/2$ (one unity more than the previous) and the sequences of
$\mathbb{F}_2^n\setminus(W_0\cup W_1\cup B)$ generate triangles of weight
strictly greater than $(3n-1)/2$. Therefore, $w_2=(3n-3)/2$ and
$W_2=\{\mathbf{b}_2,\,\mathbf{b}_4,\,\mathbf{b}_6\}$.\qed
\end{proof}

\begin{d2}[tb]
$$
\begin{array}{|l|l|l|}
\hline
\partial\mathbf{x} & \mathbf{x}  \\
\hline
\mathbf{b}_1^{(n-1)}=\overline{10}[n-1] &  \overline{0110}[n] \\
                                        &  \overline{1001}[n] \\
\hline
\mathbf{b}_2^{(n-1)}=01\cdot\overline{0}[n-3] & 00\cdot\overline{1}[n-2] \\
                                              & 11\cdot\overline{0}[n-2]=\mathbf{b}_5  \\
\hline
\mathbf{b}_3^{(n-1)}=\overline{0}[n-3]\cdot 11 & \overline{0}[n-2]\cdot 10=\mathbf{b}_6 \\
                                               & \overline{1}[n-2]\cdot 01  \\
\hline
\mathbf{b}_4^{(n-1)}=\overline{01}[n-1] & \overline{0011}[n] \\
                                        & \overline{1100}[n] \\
\hline
\mathbf{b}_5^{(n-1)}=11\cdot\overline{0}[n-3] & 01\cdot\overline{0}[n-2]=\mathbf{b}_2 \\
                                              & 10\cdot\overline{1}[n-2]  \\
\hline
\mathbf{b}_6^{(n-1)}=\overline{0}[n-3]\cdot 10 & \overline{0}[n-2]\cdot 11=\mathbf{b}_3 \\
                                               & \overline{1}[n-2]\cdot 00  \\
\hline
\end{array}
$$
\vspace{-3truemm}
\caption{Sequences $\mathbf{x}$ with $\partial\mathbf{x}\in B^{(n-1)}$.}
\label{pes x+dxb}
\end{d2}

\section{$3$-sequences}

From the results of the previous section we can easily solve the problem of determining  $w_3$ and $W_3$ when $n$ is odd.

\begin{thm}
\label{W3 n senar}
 If $n\ge 7$ is odd, then $w_3=(3n-1)/2$ and $W_3=\{\mathbf{b}_1, \,\mathbf{b}_3,\,\mathbf{b}_5\}$.
\end{thm}
\begin{proof}
  The case $n=7$ follows from Table~\ref{W3 casos 4<=n<=8}. Thus, we can assume, $n\ge 9$.
  If $n$ is odd, we have seen that $w_2=\lfloor (3n-2)/2\rfloor=(3n-3)/2$. Now,
 $|T(\mathbf{b}_1)|=|T(\mathbf{b}_3)|=T(\mathbf{b}_5)|=(3n-1)/2=(3n-3)/2+1=w_2+1$,
 and by Proposition~\ref{auxW2}, the weight of all triangles generated by sequences in $\mathbf{x}\in\mathbb{F}_2^n\setminus(W_0\cup W_1\cup B)$ satisfy
$|T(\mathbf{x})|>(3n-1)/2$. Hence $w_3=(3n-1)/2$ and $W_3=\{\mathbf{b}_1, \,\mathbf{b}_3,\, \mathbf{b}_5\}$. \qed
\end{proof}

From now on, in this section we consider the case when $n$ is even.

The primitives of $\mathbf{b}_4^{(n-1)}=\overline{01}[n-1]$ are $\mathbf{c}_1=\overline{0011}[n]$
and $\mathbf{c}_4=\overline{1100}[n]$. We shall see that the sequences of $W_3$ belong to the equivalence classes
of $\mathbf{c}_1$ and $\mathbf{c}_4$. Define
$$
\begin{array}{lll}
\mathbf{c}_1=\overline{0011}[n],
& \mathbf{c}_2=101\cdot\overline{0}[n-3],
& \mathbf{c}_3=\overline{0}[n-3]\cdot 100,\\
\mathbf{c}_4=\overline{1100}[n],
& \mathbf{c}_5=001\cdot\overline{0}[n-3],
& \mathbf{c}_6=\overline{0}[n-3]\cdot 101.\\
\end{array}
$$

\begin{figure}[htb]
\label{fci}
\centering
\setlength{\unitlength}{2.35mm}
\begin{picture}(10,9)
\n{0.5}{7.5} \n{1.5}{7.5} \y{2.5}{7.5} \y{3.5}{7.5} \n{4.5}{7.5} \n{5.5}{7.5} \y{6.5}{7.5} \y{7.5}{7.5}
\n{1.0}{6.5} \y{2.0}{6.5} \n{3.0}{6.5} \y{4.0}{6.5} \n{5.0}{6.5} \y{6.0}{6.5} \n{7.0}{6.5}
\y{1.5}{5.5} \y{2.5}{5.5} \y{3.5}{5.5} \y{4.5}{5.5} \y{5.5}{5.5} \y{6.5}{5.5}
\n{2.0}{4.5} \n{3.0}{4.5} \n{4.0}{4.5} \n{5.0}{4.5} \n{6.0}{4.5}
\n{2.5}{3.5} \n{3.5}{3.5} \n{4.5}{3.5} \n{5.5}{3.5}
\n{3.0}{2.5} \n{4.0}{2.5} \n{5.0}{2.5}
\n{3.5}{1.5} \n{4.5}{1.5}
\n{4.0}{0.5}
\put(8.5,7.5){\vector(-1,-2){3.7}} %fletxa dreta
\put(3.2,0){\vector(-1,2){3.7}} %fletxa esquerra
\put(3,8.5){$\mathbf{c}_1$}
\put(7,3.5){$r(\mathbf{c}_1)=\mathbf{c}_2$}
\put(-7,3.5){$\ell(\mathbf{c}_1)=\mathbf{c}_3$}
\end{picture}
\hspace{3truecm}
\begin{picture}(10,9)
\y{0.5}{7.5} \y{1.5}{7.5} \n{2.5}{7.5} \n{3.5}{7.5} \y{4.5}{7.5} \y{5.5}{7.5} \n{6.5}{7.5} \n{7.5}{7.5}
\n{1.0}{6.5} \y{2.0}{6.5} \n{3.0}{6.5} \y{4.0}{6.5} \n{5.0}{6.5} \y{6.0}{6.5} \n{7.0}{6.5}
\y{1.5}{5.5} \y{2.5}{5.5} \y{3.5}{5.5} \y{4.5}{5.5} \y{5.5}{5.5} \y{6.5}{5.5}
\n{2.0}{4.5} \n{3.0}{4.5} \n{4.0}{4.5} \n{5.0}{4.5} \n{6.0}{4.5}
\n{2.5}{3.5} \n{3.5}{3.5} \n{4.5}{3.5} \n{5.5}{3.5}
\n{3.0}{2.5} \n{4.0}{2.5} \n{5.0}{2.5}
\n{3.5}{1.5} \n{4.5}{1.5}
\n{4.0}{0.5}
\put(8.5,7.5){\vector(-1,-2){3.7}} %fletxa dreta
\put(3.2,0){\vector(-1,2){3.7}} %fletxa esquerra
\put(3,8.5){$\mathbf{c}_4$}
\put(7,3.5){$r(\mathbf{c}_4)=\mathbf{c}_5$}
\put(-7,3.5){$\ell(\mathbf{c}_4)=\mathbf{c}_6$}
\end{picture}
\\[10pt]
\begin{picture}(13,9)
\n{0.5}{7.5} \n{1.5}{7.5} \y{2.5}{7.5} \y{3.5}{7.5} \n{4.5}{7.5} \n{5.5}{7.5} \y{6.5}{7.5} \y{7.5}{7.5} \n{8.5}{7.5} \n{9.5}{7.5}
\n{1.0}{6.5} \y{2.0}{6.5} \n{3.0}{6.5} \y{4.0}{6.5} \n{5.0}{6.5} \y{6.0}{6.5} \n{7.0}{6.5} \y{8.0}{6.5} \n{9.0}{6.5}
\y{1.5}{5.5} \y{2.5}{5.5} \y{3.5}{5.5} \y{4.5}{5.5} \y{5.5}{5.5} \y{6.5}{5.5} \y{7.5}{5.5} \y{8.5}{5.5}
\n{2.0}{4.5} \n{3.0}{4.5} \n{4.0}{4.5} \n{5.0}{4.5} \n{6.0}{4.5} \n{7.0}{4.5} \n{8.0}{4.5}
\n{2.5}{3.5} \n{3.5}{3.5} \n{4.5}{3.5} \n{5.5}{3.5} \n{6.5}{3.5} \n{7.5}{3.5}
\n{3.0}{2.5} \n{4.0}{2.5} \n{5.0}{2.5} \n{6.0}{2.5} \n{7.0}{2.5}
\n{3.5}{1.5} \n{4.5}{1.5} \n{5.5}{1.5} \n{6.5}{1.5}
\n{4.0}{0.5} \n{5.0}{0.5} \n{6.0}{0.5}
\n{4.5}{-0.5} \n{5.5}{-0.5}
\n{5.0}{-1.5}
\put(10.5,7.5){\vector(-1,-2){4.8}} %fletxa dreta
\put(4.2,-2){\vector(-1,2){4.8}} %fletxa esquerra
\put(4.5,8.5){$\mathbf{c}_1$}
\put(9.5,3.5){$r(\mathbf{c}_1)=\mathbf{c}_5$}
\put(-7,3.5){$\ell(\mathbf{c}_1)=\mathbf{c}_3$}
\end{picture}
\hspace{3truecm}
\begin{picture}(13,9)
\y{0.5}{7.5} \y{1.5}{7.5} \n{2.5}{7.5} \n{3.5}{7.5} \y{4.5}{7.5} \y{5.5}{7.5} \n{6.5}{7.5} \n{7.5}{7.5} \y{8.5}{7.5} \y{9.5}{7.5}
\n{1.0}{6.5} \y{2.0}{6.5} \n{3.0}{6.5} \y{4.0}{6.5} \n{5.0}{6.5} \y{6.0}{6.5} \n{7.0}{6.5} \y{8.0}{6.5} \n{9.0}{6.5}
\y{1.5}{5.5} \y{2.5}{5.5} \y{3.5}{5.5} \y{4.5}{5.5} \y{5.5}{5.5} \y{6.5}{5.5} \y{7.5}{5.5} \y{8.5}{5.5}
\n{2.0}{4.5} \n{3.0}{4.5} \n{4.0}{4.5} \n{5.0}{4.5} \n{6.0}{4.5} \n{7.0}{4.5} \n{8.0}{4.5}
\n{2.5}{3.5} \n{3.5}{3.5} \n{4.5}{3.5} \n{5.5}{3.5} \n{6.5}{3.5} \n{7.5}{3.5}
\n{3.0}{2.5} \n{4.0}{2.5} \n{5.0}{2.5} \n{6.0}{2.5} \n{7.0}{2.5}
\n{3.5}{1.5} \n{4.5}{1.5} \n{5.5}{1.5} \n{6.5}{1.5}
\n{4.0}{0.5} \n{5.0}{0.5} \n{6.0}{0.5}
\n{4.5}{-0.5} \n{5.5}{-0.5}
\n{5.0}{-1.5}
\put(10.5,7.5){\vector(-1,-2){4.8}} %fletxa dreta
\put(4.2,-2){\vector(-1,2){4.8}} %fletxa esquerra
\put(4.5,8.5){$\mathbf{c}_4$}
\put(9.5,3.5){$r(\mathbf{c}_4)=\mathbf{c}_2$}
\put(-7,3.5){$\ell(\mathbf{c}_4)=\mathbf{c}_6$}
\end{picture}
\vspace{3mm}
\caption{The sequences $\mathbf{c}_i$ for $n=8\equiv 0\pmod{4}$ and $n=10\equiv 2\pmod{4}$.}
\end{figure}
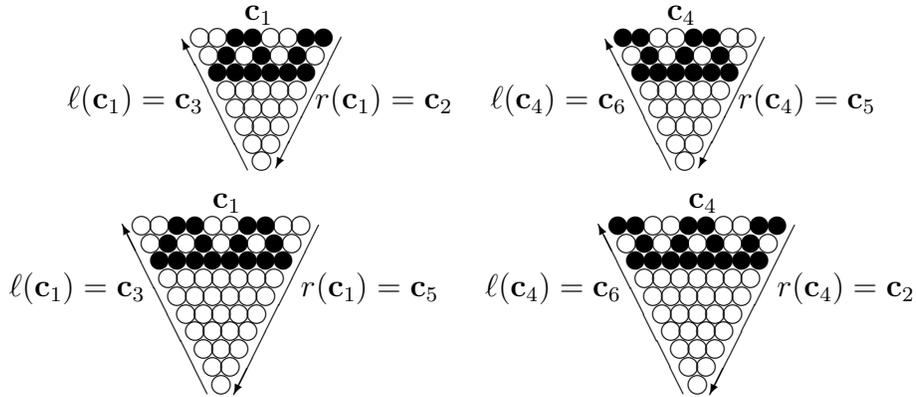

\begin{remark}
\label{simetries ci}
\normalfont
If
$n\equiv 0\pmod{4}$, then (see top of Figure~\ref{fci})
$\mathbf{c}_2=r(\mathbf{c}_1)$, $\mathbf{c}_3=\ell(\mathbf{c}_1)$,
$\mathbf{c}_4=i(\mathbf{c}_1)$, $\mathbf{c}_5=r(\mathbf{c}_4)$ and $\mathbf{c}_6=\ell(\mathbf{c}_4)$.
If
$n\equiv 2\pmod{4}$, then (see Figure~\ref{fci} down)  $\mathbf{c}_5=r(\mathbf{c}_1)$,
$\mathbf{c}_3=\ell(\mathbf{c}_1)$, $\mathbf{c}_2=r(\mathbf{c}_4)$ and
$\mathbf{c}_6=\ell(\mathbf{c}_4)$.
\end{remark}

\begin{remark}
\label{pesos ci}
\normalfont
It is clear that if $n\equiv 0\pmod{4}$, then $|\mathbf{c}_1|=|\mathbf{c}_4|=n/2$, whereas if
$n\equiv 2\pmod{4}$, then $|\mathbf{c}_1|=(n-2)/2$ i $|\mathbf{c}_4|=(n+2)/2$. Moreover, for all $n\ge 4$,
$|\mathbf{c}_2|=|\mathbf{c}_6|=2$ and $|\mathbf{c}_3|=|\mathbf{c}_5|=1$.
\end{remark}

\begin{prop} Let $n\ge 4$ be the length of the considered sequences, and assume that $n$ is even.
\label{pesos T(ci)}
\vspace{-\parskip}
\begin{enumerate}[\rm(i)]
\item If $n\equiv 0\pmod{4}$, then $|T(\mathbf{c}_i)|=2n-3$ for all $i\in[6]$.
\item  If $n\equiv 2\pmod{4}$, then
$$
|T(\mathbf{c}_1)|=|T(\mathbf{c}_3)|=|T(\mathbf{c}_5)|=2n-4,\quad
|T(\mathbf{c}_2)|=|T(\mathbf{c}_4)|=|T(\mathbf{c}_6)|=2n-2.
$$
\end{enumerate}
\end{prop}
\begin{proof}
We have $|T(\mathbf{c}_1)|=|\mathbf{c}_1|+|\partial\mathbf{c}_1|+|T(\partial^2\mathbf{c}_1)|$.
The derivative of $\mathbf{c}_1=\overline{0011}[n]$ is $\partial\mathbf{c}_1=\overline{01}[n-1]$, and
 $\partial^2\mathbf{c}_1=\overline{1}[n-2]=\mathbf{a}_1^{(n-2)}$. Therefore,
$|\partial\mathbf{c}_1|+|T(\partial^2\mathbf{c}_1)|=(n-2)/2+n-2=3n/2-3$ and $|T(\mathbf{c}_1)|=|\mathbf{c}_1|+3n/2-3$.

(i) If  $n\equiv 0\pmod{4}$, by  Remark~\ref{pesos ci} $|\mathbf{c}_1|=n/2$ and then $|T(\mathbf{c}_1)|=n/2+3n/2-3=2n-3$.
From Remark~\ref{simetries ci}, it follows $|T(\mathbf{c}_i)|=|T(\mathbf{c}_1)|=2n-3$
for all $i\in[6]$.

(ii) If  $n\equiv 2\pmod{4}$, by  Remark~\ref{pesos ci} $|\mathbf{c}_1|=(n-2)/2$ and then $|T(\mathbf{c}_1)|=(n-2)/2+3n/2-3=2n-4$.
From Remark~\ref{simetries ci}, it follows $|T(\mathbf{c}_3)|=|T(\mathbf{c}_5)=|T(\mathbf{c}_1)|$.

We make the same argument for per a $\mathbf{c}_4=\overline{1100}[n]$. The only difference is that, now,
$|\mathbf{c}_4|=(n+1)/2$. We obtain, $|T(\mathbf{c}_4)|=(n+1)/2+3n/2-3=2n-2$. By Remark~\ref{simetries ci},
$|T(\mathbf{c}_2)|=|T(\mathbf{c}_6)=|T(\mathbf{c}_4)|$.\qed
\end{proof}

Define $C=C^{(n)}=\{\mathbf{c}_i: i\in[6]\}\}$.

The inductive argument we use is not valid for $n=8$ and $n=10$, so we take $n=12$ as the basic case.
The particular (computer-aided) study of cases $n=8$, $n=10$ and the basic case $n=12$
give the results summarized in Proposition~\ref{W3 n<=8<=12}.

\begin{prop}\phantom{x}
\label{W3 n<=8<=12}
\begin{enumerate}[\rm (i)]
\item $w_3^{(8)}=13$ and
\begin{align*}
W_3^{(8)}=\{&\mathbf{c}_1,\, \mathbf{c}_2,\,  \mathbf{c}_3,\, \mathbf{c}_4,\, \mathbf{c}_5,\, \mathbf{c}_6,\\
           & 11110000, \, 00001000, \, 00010001, \, 00001111, \, 10001000, \, 00010000\}.
\end{align*}
\item $w_3^{(10)}=16$ and $W_3^{(10)}=\{\mathbf{c}_1, \, \mathbf{c}_3,  \, \mathbf{c}_5\}$.
\item $w_3^{(12)}=21$ and
$W_3^{(12)}=\{\mathbf{c}_1,\, \mathbf{c}_2,\,  \mathbf{c}_3,\, \mathbf{c}_4,\, \mathbf{c}_5,\, \mathbf{c}_6\}$.
\end{enumerate}
\end{prop}

\begin{remark}\normalfont
The sequences of $W_3^{(8)}$ different from $\mathbf{c}_i$, form the equivalence class of
$\mathbf{g}=11110000$. Indeed, in the list above, they are $\mathbf{g}$,
$r(\mathbf{g})$, $\ell(\mathbf{g})$, $i(\mathbf{g})$, $r(i(\mathbf{g}))$ and $\ell(i(\mathbf{g}))$,
respectively.
\end{remark}

\begin{remark}\normalfont
\label{n=10, 2n-3}
For $n=10$, there exist six equivalent sequences that generate triangles of weight $2n-3=17$. They are
$$
\begin{array}{lll}
\mathbf{x}=0001000000,   & r(\mathbf{x})=0000001100, & \ell(\mathbf{x})=0010001000,\\
i(\mathbf{x})=0000001000, & i(r(\mathbf{x}))=0011000000, & i(\ell(\mathbf{x}))=0001000100.
\end{array}
$$
But, for $n=14$ no sequence $\mathbf{x}$ exists with $|T(\mathbf{x})|=25=2n-3$.
\end{remark}

\begin{prop}
\label{auxW3}
Let $n\ge 12$ be even and $\mathbf{x}\in \mathbb{F}_2^{n}\setminus(W_0\cup W_1\cup W_2\cup C)$. Then
$|T(\mathbf{x})|>2n-3$.
\end{prop}
\begin{proof}
  By induction on $n\ge 12$ even. For $n=12$, we have seen that
  $w_3^{(12)}=21=2n-3$ and that the sequences that generate triangles of weight
  21 are, exactly, those in $C^{(12)}$. Let $n\ge 14$ be even and let us suppose the
  proposition true for even values lower or equal to $n-2$ (and greater than or equal to 12).

Let $\mathbf{x}\in \mathbb{F}_2^{n}\setminus(W_0\cup W_1\cup W_2\cup C)$.
First, we consider the case $\partial^2\mathbf{x}\not\in W_0^{(n-2)}\cup W_1^{(n-2)}\cup W_2^{(n-2)}\cup C^{(n-2)}$.

We have  $|\mathbf{x}|\ge 1$ because $\mathbf{x}\not\in W_0$.
If $|\mathbf{x}|=1$, then $\mathbf{x}=\mathbf{e}_k$ for some $k$ which, by symmetry, we can assume to be $k\le (n-1)/2$.
Since $\mathbf{e}_0=\mathbf{a_2}\in W_1$, $\mathbf{e}_1=\mathbf{b}_2\in W_2$, and $\mathbf{e}_2=\mathbf{c}_5\in C$,
it must be $k\ge 3$. From Proposition~\ref{pesos T(e0),T(e1),T(e2)}, for $k=3$ we have
$|T(\mathbf{x})|=|T(\mathbf{e}_3)|=\lfloor(9n-20)/4\rfloor>2n-3$. For $k\ge 4$, by Proposition~\ref{pesos T(ei)}
we have $|T(\mathbf{x})|=|T(\mathbf{e}_k)|>2n-3$.

Thus, we can assume that $|\mathbf{x}|\ge 2$. By induction hypothesis, $|T(\partial^2\mathbf{x})|>2(n-2)-3=2n-7$.
Since,
$$
|T(\mathbf{x})|=|\mathbf{x}|+|\partial\mathbf{x}|+|T(\partial^2\mathbf{x})|
>|\mathbf{x}|+|\partial\mathbf{x}|+2n-7,
$$
it is sufficient to see that $s=|\mathbf{x}|+|\partial\mathbf{x}|\ge 4$.

If $|\mathbf{x}|\ge 4$, then obviously $s\ge 4$. If $|\mathbf{x}|=3$,
then $|\partial\mathbf{x}|\ge 1$ and $s\ge 4$.  If
$|\mathbf{x}|=2$, as $\mathbf{x}\ne \mathbf{b}_5$ and
$\mathbf{x}\ne \mathbf{b}_3$, we see that $\mathbf{x}$ contains the sequence
$0110$ or two non-consecutive ones; in both cases,
$|\partial\mathbf{x}|\ge 2$, and  $s\ge 4$, too.

Thus, if $\partial^2\mathbf{x}\not\in W_0^{(n-2)}\cup W_1^{(n-2)}\cup W_2^{(n-2)}\cup C^{(n-2)}$, then
$|T(\mathbf{x})|>2n-3$. Now let us consider the case
$\partial^2\mathbf{x}\in W_0^{(n-2)}\cup W_1^{(n-2)}\cup W_2^{(n-2)}\cup C^{(n-2)}$.

It can be checked in Table~\ref{d2x=0} that if $\partial^2\mathbf{x}\in W_0$, then $\mathbf{x}\in W_0\cup W_1\cup W_2$,
which contradicts the hypothesis on $\mathbf{x}$. Then, $\partial^2\mathbf{x}\not\in W_0$.

\begin{d2}[H]
$$
\begin{array}{|l|l|l|}
\hline
\partial^2\mathbf{x} & \partial\mathbf{x} & \mathbf{x} \\
\hline
\mathbf{0}^{(n-2)} & \mathbf{0}^{(n-1)} &\mathbf{0}^{(n)}\\
                   &                    &\mathbf{1}^{(n)}=\mathbf{a}_1\\
                   & \mathbf{1}^{(n-1)} &\overline{10}[n]=\mathbf{b}_1\\
                   &                    &\overline{01}[n]=\mathbf{b}_2  \\
\hline
\end{array}
$$
\caption{Sequences $\mathbf{x}$ such that $\partial^2\mathbf{x}\in W_0=\{\mathbf{0}\}$.}
\label{d2x=0}
\end{d2}

Suppose that $\partial^2\mathbf{x}=\mathbf{a}_i^{(n-2)}\in W_1$. From
Remark~\ref{pesos T(ai)}, we have
$|T(\partial^2\mathbf{x})|=|T(\mathbf{a}_i^{(n-2)})|=n-2$. If we prove that
 $|\mathbf{x}|+|\partial\mathbf{x}|>n-1$, then
$|T(\mathbf{x})|=|\mathbf{x}|+|\partial\mathbf{x})|+n-1>2n-2>2n-3$, as wanted.
Table~\ref{d2x=ai}\footnote{Table~\ref{d2x=ai} is in the last section, together with
other tables of large size.} shows, for each
$\mathbf{a}_i^{(n-2)}\in W_1^{(n-2)}$, the four possibles sequences $\mathbf{x}$ such that
$\partial^2\mathbf{x}=\mathbf{a}_i^{(n-2)}$. For those $\mathbf{x}\not\in W_0\cup
W_1\cup W_2\cup C$, a lower bound of the sum
$|\mathbf{x}|+|\partial\mathbf{x}|$ is given; in all cases the bound is $>n-1$.

Suppose that $\partial^2\mathbf{x}=\mathbf{b}_i^{(n-2)}\in
W_2^{(n-2)}$. By Proposition~\ref{pesos T(bi)} (and because $n$ is
even),
$|T(\partial^2\mathbf{x})|=|T(\mathbf{b}_i^{(n-2)})|=(3n-8)/2$. If we
prove that $|\mathbf{x}|+|\partial\mathbf{x}|>(n+2)/2$, then
$|T(\mathbf{x})|>2n-3$, as wanted. Table~\ref{d2x=bi} shows, for each
$\mathbf{b}_i^{(n-2)}\in W_2^{(n-2)}$ the four possibles $\mathbf{x}$
such that $\partial^2\mathbf{x}=\mathbf{b}_i^{(n-2)}$. For
$\mathbf{x}\not\in W_0\cup W_1\cup W_2\cup C$, a lower bound of the
sum $|\mathbf{x}|+|\partial\mathbf{x}|$ is given. In all cases the sum
is $>(n+2)/2$.

Finally, suppose that  $\partial^2\mathbf{x}\in C^{(n-2)}$. By Proposition~\ref{pesos T(ci)} (and because $n$ is even),
$|T(\partial^2\mathbf{x})|=|T(\mathbf{c}_i^{(n-2)})|\ge 2(n-2)-4=2n-8$.  If
we prove that $|\mathbf{x}|+|\partial\mathbf{x}|>5$, then
$|T(\mathbf{x})|>2n-3$, as wanted. Table~\ref{d2x=ci} shows, for each
$\mathbf{c}_i^{(n-2)}\in C^{(n-2)}$ the four possible $\mathbf{x}$ such that
$\partial^2\mathbf{x}=\mathbf{c}_i^{(n-2)}$. For $\mathbf{x}\not\in W_0\cup
W_1\cup W_2\cup C$, it is clear that
$|\mathbf{x}|+|\partial\mathbf{x}|\ge 6$. \qed
\end{proof}

As a conclusion of the section, we have the following theorem:

\begin{thm} Let $n\ge 10$ be an even integer.
\vspace{-\parskip}
\begin{enumerate}[\rm (i)]
\item If $n\equiv 0\pmod{4}$,  then $w_3=2n-3$ and
$W_3=\{\mathbf{c}_1,\,\mathbf{c}_2,\,\mathbf{c}_3,\,\mathbf{c}_4,\,\mathbf{c}_5,\,\mathbf{c}_6\}$.
\item If $n\equiv 2\pmod{4}$,  then $w_3=2n-4$ and
$W_3=\{\mathbf{c}_1,\,\mathbf{c}_3,\,\mathbf{c}_5 \}$.
\end{enumerate}
\end{thm}
\begin{proof}
(i) By Proposition~\ref{pesos T(ci)}, the sequences in $C=\{\mathbf{c}_i: i\in[6]\}$ generate
triangles of weight $2n-3$, and by Proposition~\ref{auxW3} all sequences in
$\mathbb{F}_2^n\setminus (W_0\cup W_1\cup W_2\cup C)$ generate triangles of weight $>2n-3$.
Therefore, $w_3=2n-3$ and $W_3=C$.

(ii) By Proposition~\ref{pesos T(ci)}, sequences $\mathbf{c}_1$, $\mathbf{c}_3$, and $\mathbf{c}_5$  generate
triangles of weight $2n-4$, sequences  $\mathbf{c}_2$, $\mathbf{c}_4$, and $\mathbf{c}_6$ generate triangles of weight
$2n-2>2n-4$ and, by Proposition~\ref{auxW3}, all sequences
in $\mathbb{F}_2^n\setminus (W_0\cup W_1\cup W_2\cup C)$ generate triangles of weight $>2n-3>2n-4$. Then,
$w_3=2n-4$ and $W_3=\{\mathbf{c}_1,\,\mathbf{c}_3,\,\mathbf{c}_5 \}$. \qed
\end{proof}

\section{$m$-sequences}

H.~Harborth~\cite{Harborth} observed that $w_m=\lceil n(n+1)/3\rceil$ and gave sequences that generate triangles
with this weight; that is, sequences in $W_m$. In this section we determine the set $W_m$ and re-obtain the value of $w_m$.

For $n\ge 2$, define
$$
\mathbf{z}_1=\overline {110}[n],\quad
\mathbf{z}_2=\overline{011}[n], \quad
\mathbf{z}_3=\overline{101}[n].
$$
and $Z=Z^{(n)}=\{\mathbf{z}_1,\, \mathbf{z}_2,\, \mathbf{z}_3\}$.

\begin{remark}
\normalfont
If $n\equiv 0\pmod{3}$, then $\mathbf{z}_2=r(\mathbf{z}_1)$ and $\mathbf{z}_3=\ell(\mathbf{z}_1)$
(see Figure~\ref{fzi}, top left).
If $n\equiv 2\pmod{3}$, then $\mathbf{z}_2=\ell(\mathbf{z}_1)$ and $\mathbf{z}_3=r(\mathbf{z}_1)$
(see Figure~\ref{fzi}, top right).
If $n\equiv 1\pmod{3}$, then $\mathbf{z}_i=r(\mathbf{z}_i)=\ell(\mathbf{z}_i)$ for $i\in[3]$, and
$\mathbf{z}_3=i(\mathbf{z}_1)$ (see Figure~\ref{fzi} bottom).
\end{remark}

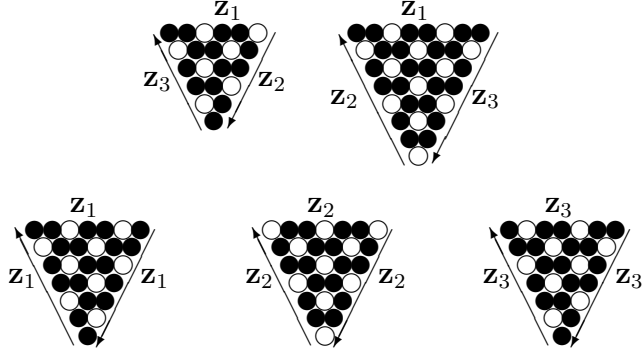
\begin{figure}[htb]
\centering
\setlength{\unitlength}{2.35mm}
\begin{picture}(10,10)
\y{0.5}{7.5} \y{1.5}{7.5} \n{2.5}{7.5} \y{3.5}{7.5} \y{4.5}{7.5} \n{5.5}{7.5}
\n{1.0}{6.5} \y{2.0}{6.5} \y{3.0}{6.5} \n{4.0}{6.5} \y{5.0}{6.5}
\y{1.5}{5.5} \n{2.5}{5.5} \y{3.5}{5.5} \y{4.5}{5.5}
\y{2.0}{4.5} \y{3.0}{4.5} \n{4.0}{4.5}
\n{2.5}{3.5} \y{3.5}{3.5}
\y{3.0}{2.5}
\put(6.5,7.5){\vector(-1,-2){2.7}} %fletxa dreta
\put(2.3,2){\vector(-1,2){2.7}} %fletxa esquerra
\put(3,8.5){$\mathbf{z}_1$}
\put(5.5,4.5){$\mathbf{z}_2$}
\put(-1,4.5){$\mathbf{z}_3$}
\end{picture}
%\hspace{1truecm}
\begin{picture}(8,10)
\y{0.5}{7.5} \y{1.5}{7.5} \n{2.5}{7.5} \y{3.5}{7.5} \y{4.5}{7.5} \n{5.5}{7.5} \y{6.5}{7.5} \y{7.5}{7.5}
\n{1.0}{6.5} \y{2.0}{6.5} \y{3.0}{6.5} \n{4.0}{6.5} \y{5.0}{6.5} \y{6.0}{6.5} \n{7.0}{6.5}
\y{1.5}{5.5} \n{2.5}{5.5} \y{3.5}{5.5} \y{4.5}{5.5} \n{5.5}{5.5} \y{6.5}{5.5}
\y{2.0}{4.5} \y{3.0}{4.5} \n{4.0}{4.5} \y{5.0}{4.5} \y{6.0}{4.5}
\n{2.5}{3.5} \y{3.5}{3.5} \y{4.5}{3.5} \n{5.5}{3.5}
\y{3.0}{2.5} \n{4.0}{2.5} \y{5.0}{2.5}
\y{3.5}{1.5} \y{4.5}{1.5}
\n{4.0}{0.5}
\put(8.5,7.5){\vector(-1,-2){3.7}} %fletxa dreta
\put(3.2,0){\vector(-1,2){3.7}} %fletxa esquerra
\put(3,8.5){$\mathbf{z}_1$}
\put(7,3.5){$\mathbf{z}_3$}
\put(-1,3.5){$\mathbf{z}_2$}
\end{picture}
\\
\begin{picture}(8,10)
\y{1.0}{6.5} \y{2.0}{6.5} \n{3.0}{6.5} \y{4.0}{6.5} \y{5.0}{6.5} \n{6.0}{6.5} \y{7.0}{6.5}
\n{1.5}{5.5} \y{2.5}{5.5} \y{3.5}{5.5} \n{4.5}{5.5} \y{5.5}{5.5} \y{6.5}{5.5}
\y{2.0}{4.5} \n{3.0}{4.5} \y{4.0}{4.5} \y{5.0}{4.5} \n{6.0}{4.5}
\y{2.5}{3.5} \y{3.5}{3.5} \n{4.5}{3.5} \y{5.5}{3.5}
\n{3.0}{2.5} \y{4.0}{2.5} \y{5.0}{2.5}
\y{3.5}{1.5} \n{4.5}{1.5}
\y{4.0}{0.5}
\put(7.8,6.5){\vector(-1,-2){3.3}} %fletxa dreta
\put(3.2,0){\vector(-1,2){3.3}} %fletxa esquerra
\put(3,7.5){$\mathbf{z}_1$}
\put(7,3.5){$\mathbf{z}_1$}
\put(-0.5,3.5){$\mathbf{z}_1$}
\end{picture}
\hspace{1truecm}
\begin{picture}(8,10)
\n{1.0}{6.5} \y{2.0}{6.5} \y{3.0}{6.5} \n{4.0}{6.5} \y{5.0}{6.5} \y{6.0}{6.5} \n{7.0}{6.5}
\y{1.5}{5.5} \n{2.5}{5.5} \y{3.5}{5.5} \y{4.5}{5.5} \n{5.5}{5.5} \y{6.5}{5.5}
\y{2.0}{4.5} \y{3.0}{4.5} \n{4.0}{4.5} \y{5.0}{4.5} \y{6.0}{4.5}
\n{2.5}{3.5} \y{3.5}{3.5} \y{4.5}{3.5} \n{5.5}{3.5}
\y{3.0}{2.5} \n{4.0}{2.5} \y{5.0}{2.5}
\y{3.5}{1.5} \y{4.5}{1.5}
\n{4.0}{0.5}
\put(7.8,6.5){\vector(-1,-2){3.3}} %fletxa dreta
\put(3.2,0){\vector(-1,2){3.3}} %fletxa esquerra
\put(3,7.5){$\mathbf{z}_2$}
\put(7,3.5){$\mathbf{z}_2$}
\put(-0.5,3.5){$\mathbf{z}_2$}
\end{picture}
\hspace{1truecm}
\begin{picture}(8,10)
\y{1.0}{6.5} \n{2.0}{6.5} \y{3.0}{6.5} \y{4.0}{6.5} \n{5.0}{6.5} \y{6.0}{6.5} \y{7.0}{6.5}
\y{1.5}{5.5} \y{2.5}{5.5} \n{3.5}{5.5} \y{4.5}{5.5} \y{5.5}{5.5} \n{6.5}{5.5}
\n{2.0}{4.5} \y{3.0}{4.5} \y{4.0}{4.5} \n{5.0}{4.5} \y{6.0}{4.5}
\y{2.5}{3.5} \n{3.5}{3.5} \y{4.5}{3.5} \y{5.5}{3.5}
\y{3.0}{2.5} \y{4.0}{2.5} \n{5.0}{2.5}
\n{3.5}{1.5} \y{4.5}{1.5}
\y{4.0}{0.5}
\put(7.8,6.5){\vector(-1,-2){3.3}} %fletxa dreta
\put(3.2,0){\vector(-1,2){3.3}} %fletxa esquerra
\put(3,7.5){$\mathbf{z}_3$}
\put(7,3.5){$\mathbf{z}_3$}
\put(-0.5,3.5){$\mathbf{z}_3$}
\end{picture}
\caption{Triangles generated by the sequences $\mathbf{z}_1$, $\mathbf{z}_2$ and $\mathbf{z}_3$.}
\label{fzi}
\end{figure}

If $q$ is a rational number, $\langle q\rangle$ stands by the integer
closest to $q$ (when we use this notation, this integer will be
unique).

\begin{remark}
\label{pesos m}
\normalfont
The weight of the sequences $\mathbf{z}_i$ are easily obtained from the definition:
$$
|\mathbf{z}_1|=\lceil 2n/3\rceil, \quad 
|\mathbf{z}_2|=\lfloor 2n/3\rfloor, \quad 
|\mathbf{z}_3|=\langle 2n/3\rangle.
$$
 Thus,
\begin{enumerate}[(i)]
\vspace{-1.5\parskip}
\item If $n\equiv 0 \pmod{3}$, then
$|\mathbf{z}_1|=|\mathbf{z}_2|=|\mathbf{z}_3|=2n/3$;
\item if $n\equiv 1\pmod{3}$, then
$|\mathbf{z}_1|=|\mathbf{z}_3|=(2n+1)/3$, and  $|\mathbf{z}_2|=(2n-2)/3$;
\item if $n\equiv 2\pmod{3}$, then
$|\mathbf{z}_1|=(2n+2)/3$, and $|\mathbf{z}_2|=|\mathbf{z}_3|=(2n-1)/3$.
\end{enumerate}
\end{remark}

\begin{remark}
\label{derivades z}
\normalfont
Note that, for $n\ge 3$,
$
\partial\mathbf{z}_1^{(n)}=\mathbf{z}_2^{(n-1)},\
\partial\mathbf{z}_2^{(n)}=\mathbf{z}_3^{(n-1)}, \
\partial\mathbf{z}_3^{(n)}=\mathbf{z}_1^{(n-1)}.
$
Moreover,
\begin{align*}
&|\mathbf{z}_1|+|\partial\mathbf{z}_1|+|\partial^2\mathbf{z}_1|
=|\mathbf{z}_1^{(n)}|+|\mathbf{z}_2^{(n-1)}|+|\mathbf{z}_3^{(n-2)}|=2n-2,\\
&|\mathbf{z}_2|+|\partial\mathbf{z}_2|+|\partial^2\mathbf{z}_2|
=|\mathbf{z}_2^{(n)}|+|\mathbf{z}_3^{(n-1)}|+|\mathbf{z}_1^{(n-2)}|=2n-2,\\
&|\mathbf{z}_3|+|\partial\mathbf{z}_3|+|\partial^2\mathbf{z}_3|
=|\mathbf{z}_3^{(n)}|+|\mathbf{z}_1^{(n-1)}|+|\mathbf{z}_2^{(n-2)}|=2n-2.
\end{align*}
\end{remark}

First, we calculate the weights of the triangles $|T(\mathbf{z}_1)|$,  $|T(\mathbf{z}_2)|$ and  $|T(\mathbf{z}_3)|$.
Afterwards, we shall see that  $|T(\mathbf{x})|<|T(\mathbf{z})|$ for all $\mathbf{x}\notin Z$, $\mathbf{z}\in Z$
and $n\ge 5$.

\begin{prop} Let $n\ge 3$ be an integer.
\label{pesos T(z)}
\vspace{-\parskip}
\begin{enumerate}[\rm (i)]
\item If $n\equiv 0,2\pmod{3}$, then $|T(\mathbf{z}_i)|=n(n+1)/3$ for all $i\in[3]$.
\item If $n\equiv 1\pmod 3$, then $|T(\mathbf{z}_1)|=|T(\mathbf{z}_3)|=(n-1)(n+2)/3+1$ and
$|T(\mathbf{z}_2)|=(n-1)(n+2)/3$.
\end{enumerate}
\end{prop}
\begin{proof}
By induction on $n$. For $n\in\{3,4,5\}$ we have: $|T(\mathbf{z}_i^{(3)})|=4=n(n+1)/3$ for all $i\in[3]$;
 $|T(\mathbf{z}_1^{(4)})|=|T(\mathbf{z}_3^{(4)})|=7=(n-1)(n+2)/3+1$, $|T(\mathbf{z}_2^{(4)})|=6=(n-1)(n+2)/3$;
and $|T(\mathbf{z}_i^{(5)})|=10=n(n+1)/3$ for all $i\in[3]$. Then, for $n\in\{3,4,5\}$ the proposition holds. Assume
$n\ge 6$ and the result is true for $n-1$.

(i) Let $n\equiv 0\pmod{3}$. Then,  $|\mathbf{z}_1|=2n/3$,  $n-1\equiv 2\pmod{3}$ and
$|T(\partial\mathbf{z}_1)|=|T(\mathbf{z}_2^{(n-1)})|=(n-1)n/3$.
Hence,
$$
|T(\mathbf{z}_1)|=|\mathbf{z}_1|+|T(\partial\mathbf{z}_1)|
=2n/3+(n-1)n/3=n(n+1)/3.
$$
Now, $\mathbf{z}_2=r(\mathbf{z}_1)$, $\mathbf{z}_3=\ell(\mathbf{z}_1)$. It follows that
$|T(\mathbf{z}_1)|=|T(\mathbf{z}_2)|=|T(\mathbf{z}_3)|=n(n+1)/3$.

If $n\equiv 2\pmod{3}$, we have $|\mathbf{z}_1|=(2n+2)/3$ and $n-1\equiv 1\pmod{3}$ and
$|T(\partial\mathbf{z}_1)|=|T(\mathbf{z}_2^{(n-1)})|=(n-2)(n+1)/3$.
Therefore, $|T(\mathbf{z}_1)|=(2n+2)/3+(n-2)(n+1)/3=n(n+1)/3$.
The sequences $\mathbf{z}_1$, $\mathbf{z}_2$ and $\mathbf{z}_3$ are equivalent,
so they generate triangles of the same weight.

(ii) If $n\equiv 1\pmod{3}$, we have  $|\mathbf{z}_1|=(2n+1)/3$ and $n-1\equiv 0\pmod{3}$ and
$|T(\partial\mathbf{z}_1)|=|T(\mathbf{z}_2^{(n-1)})|=(n-1)n/3$.
Hence, $|T(\mathbf{z}_1)|=(2n+1)/3+(n-1)n/3=(n-1)(n+2)/3+1$. Moreover, $\mathbf{z}_3=i(\mathbf{z}_1)$ implies
that $|T(\mathbf{z}_3)|=|T(\mathbf{z}_1)|=(n-1)(n+2)/3+1$.

It remains to calculate  $|T(\mathbf{z}_2)|$ when $n\equiv 1\pmod{3}$. In this case, $|\mathbf{z}_2|=(2n-2)/3$ and
$n-1\equiv 0\pmod{3}$ and $|T(\partial\mathbf{z}_2)|=|T(\mathbf{z}_3^{(n-1)})|=(n-1)n/3$.
Therefore, $|T(\mathbf{z}_2)|=(2n-2)/3+(n-1)n/3=(n-1)(n+2)/3$. \qed
\end{proof}

\begin{remark}\normalfont
\label{pesos T(z) bis} The weights of triangles $T(\mathbf{z}_i)$ given in Proposition~\ref{pesos T(z)} 
can be formulated in the following alternative way:
$$
|T(\mathbf{z}_1)|=|T(\mathbf{z}_3)|=\lceil n(n+1)/3\rceil,\quad |T(\mathbf{z}_2)|=\lfloor n(n+1)/3\rfloor.
$$
\end{remark} 
To see that, for $n\ge 5$, if $\mathbf{x}\in\mathbb{F}_2^n\setminus Z$ and $\mathbf{z}\in Z$ then
$|T(\mathbf{x})|<|T(\mathbf{z})|$, we begin by considering the weight of the first three
rows of triangles of size $n$.  For a sequence $\mathbf{x}\in\mathbb{F}_2^n$,
we define $s_3(\mathbf{x})=|\mathbf{x}|+|\partial\mathbf{x}|+|\partial^2\mathbf{x}|$.
Our first goal is to prove that
$s_3(\mathbf{x})\le s_3(\mathbf{z})$
for all $\mathbf{x}\in\mathbb{F}_2^n$ and $\mathbf{z}\in Z$. Recall that, from Remark~\ref{derivades z},
we have $s_3(\mathbf{z})=2n-2$ for all $\mathbf{z}\in Z$.

The proof of the following lemma is a direct verification.

\begin{lemma}
\label{s3x n45}
Let $n\in\{4,5\}$ and $\mathbf{x}\in\mathbb{F}_2^n$.
\begin{enumerate}[\rm (i)]
\item
If $n=4$, then  $s_3(\mathbf{x})\le 6$. Moreover,
$s_3(\mathbf{x})=6$ if, and only if,
$$
\mathbf{x}\in\{(1,1,0,1), \, (0,1,1,0), \, (1,0,1,1), \, (1,0,0,1)\}.
$$
\item
If $n=5$, then $s_3(\mathbf{x})\le 8$. Moreover,
$s_3(\mathbf{x})=8$ if, and only if,
$$
\mathbf{x}\in\{(1,1,0,1,1), \, (0,1,1,0,1),\, (1,0,1,1,0),\, (1,0,0,1,1),\, (1,1,0,0,1)\}.
$$
% \item
% If $\mathbf{x}=(0,0,x_2,x_3)\in\mathbb{F}_2^4$, then $s_3(\mathbf{x})\le 5$. Moreover,
% $s_3(\mathbf{x})=5$ if and only if $\mathbf{x}=(0,0,1,1)$.
% \item
% If $\mathbf{x}=(0,0,x_2,x_3,x_4)\in\mathbb{F}_2^5$, then $s_3(\mathbf{x})\le 7$. Moreover,
% $s_3(\mathbf{x})= 7$   if and only if $\mathbf{x}=(0,0,1,1,0)\}$.
\end{enumerate}
\end{lemma}

\begin{prop}
\label{s3x}
Let $n\ge 4$ be an integer. Then $s_3(\mathbf{x})\le 2n-2$ for all $\mathbf{x}\in\mathbb{F}_2^n$.
\end{prop}
\begin{proof}
  Lemma~\ref{s3x n45} states the result for $n\in\{4,5\}$. Thus, we
  can assume $n\ge 6$.  Consider first the case $n\equiv 0\pmod{3}$, and let
  $\mathbf{x}=(x_0,x_1,\ldots,x_{n-1})\in\mathbb{F}_2^n$.  Consider
  the $q=n/3$ Steinhaus triangles $S_0=T(x_0x_1x_2)$,
  $S_1=T(x_3x_4x_5)$, $\ldots$, $S_{q-1}=T(x_{n-3}x_{n-2}x_{n-1})$.
  For $i\in\{0,\ldots, q-2\}$, between two consecutive triangles $S_i=T(x_{3i}x_{3i+1}x_{3i+2})$
  and $S_{i+1}=T(x_{3i+3}x_{3i+4}x_{3i+5})$ there is a smaller triangle $t_i$  with two rows (not a Steinhaus triangle)
  formed by the entry $x_{3i+2}+x_{3i+3}$ in the first row and
  $x_{3i+1}+x_{3i+3}$ and $x_{3i+2}+x_{3i+4}$ in the second row (see Figure~\ref{fSiti}).
  Let $S_i\cup t_i$ be the sequence formed by all entries of $S_i$ and $t_i$; and analogously
  we define $S_i\cup t_i\cup S_{i+1}$ and $S_i\cup t_i\cup S_{i+1}\cup t_{i+1}$.

\begin{figure}[htb]
\centering
\setlength{\unitlength}{1mm}
\begin{picture}(154,30)
%row 1
\put(0,25){$x_{3i}$} \put(28,25){$x_{3i+1}$} \put(56,25){$x_{3i+2}$}
\put(84,25){$x_{3i+3}$} \put(112,25){$x_{3i+4}$} \put(140,25){$x_{3i+5}$}
%row 2
\put(7,20){$x_{3i}+x_{3i+1}$} \put(35,20){$x_{3i+1}+x_{3i+2}$}\put(63,20){$x_{3i+2}+x_{3i+3}$}
\put(91,20){$x_{3i+3}+x_{3i+4}$} \put(119,20){$x_{3i+4}+x_{3i+5}$}
%row 3
\put(23,13){$x_{3i}+x_{3i+2}$}\put(51,13){$x_{3i+1}+x_{3i+3}$}\put(79,13){$x_{3i+2}+x_{3i+4}$}\put(107,13){$x_{3i+3}+x_{3i+5}$}
%S_i and t_i
\put(29,5){$S_i$}\put(72,5){$t_i$}\put(116,5){$S_{i+1}$}
%lines
\put(-2,27){\line(0,-1){4}}
\put(-2,23){\line(1,0){8}}
\put(6,23){\line(0,-1){7}}
\put(6,16){\line(1,0){15}}
\put(21,16){\line(0,-1){6}}
\put(21,10){\line(1,0){23}}
\put(44,10){\line(0,1){6}}
\put(44,16){\line(1,0){16}}
\put(60,16){\line(0,1){7}}
\put(60,23){\line(1,0){7}}
\put(67,23){\line(0,1){4}}

\put(82,27){\line(0,-1){4}}
\put(82,23){\line(1,0){8}}
\put(90,23){\line(0,-1){7}}
\put(90,16){\line(1,0){15}}
\put(106,16){\line(0,-1){6}}
\put(106,10){\line(1,0){25}}
\put(131,10){\line(0,1){6}}
\put(131,16){\line(1,0){13}}
\put(144,16){\line(0,1){7}}
\put(144,23){\line(1,0){7}}
\put(151,23){\line(0,1){4}}
\end{picture}
\vspace{-5truemm}
\caption{Triangles $S_i$, $t_i$ and $S_{i+1}$.}
\label{fSiti}
\end{figure}
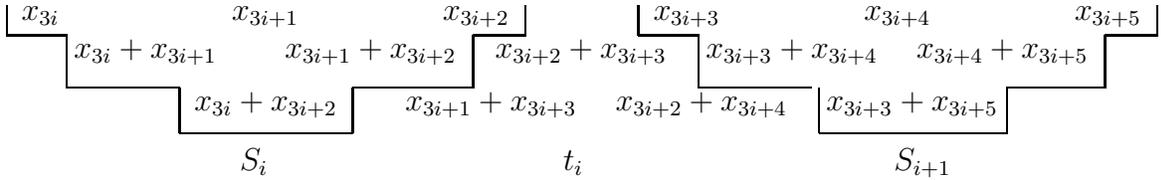

For $i\in\{0,\ldots,q-1\}$,  we have $|S_i|\le 4$, and $|S_i|=4$ if and only if
  $x_ix_{i+1}x_{i+2}\in Z^{(3)}=\{110,\, 011,\, 101\}$.  If $i\le q-2$ and $|S_i|<4$, then
  $|S_i\cup t_i|\le 6$. If $|S_i|=4$ and $|t_i|\le 2$, then $|S_i\cup
  t_i|\le 6$, too.  Consider the case $|S_i|=4$ and $|t_i|=3$. The
  unique sequence in $Z^{(3)}$ that can produce a triangle $t_i$ with $|t_i|=3$
  is $011$; that is $S_i=T(011)$. Then, we have $x_{3i+3}=0$,
  $x_{3i+3}+x_{3i+4}=0$, and $x_{3i+4}=0$.  Hence, $|S_{i+1}|\le
  3$. Then $|S_i\cup t_i\cup S_{i+1}|\le 4+3+3=10$. If, moreover,
  $|t_{i+1}|=3$, then $|S_{i+1}|=0$ and, as a consequence, $|S_i\cup t_i|\le 7$,
  $|S_{i+1}\cup t_{i+1}|=3$ and $|S_i\cup t_i \cup S_{i+1}\cup
  t_{i+1}|\le 10<12$.  Therefore, an upper bound of
  $s_3(\mathbf{x})$ is
  obtained by counting as if each of the $q=n/3$ triangles $S_i$ were of
  weight $4$ and each of the $q-1=(n-3)/3$ triangles $t_i$ were of
  weight $2$; that is,
  $s_3(x)\le 4n/3+2(n-3)/3=2n-2$.

  Consider now the case $n\equiv 1\pmod{3}$. We follow a similar
  notation as above, but taking as $S_0$ the first three rows of
  $T(x_0x_1x_2x_3)$. We first study the concrete case $n=7$; we have
  $2n-2=12$.  As before, let $t_0=(x_3+x_4, x_2+x_4, x_3+x_5)$ and
  $S_1=T(x_4x_5x_6)$.  If $|S_0|\le 5$, then $s_3(\mathbf{x})=|S_0\cup
  t_0\cup S_1|\le 5+3+4=12$ and the result holds. So we can assume
  that $|S_0|$ has the maximum value $|S_0|=6$. If $|t_0|\le 2$, then
  $s_3(\mathbf{x})=|S_0\cup t_0\cup S_1|\le 6+2+4=12$, and we are
  done. Only the case $|S_0|=6$ and $|t_0|=3$ remains. These
  conditions imply that $x_0x_1x_2x_3=1011$, in particular
  $x_3=1$. Since $|t_0|=3$, we have $x_4=x_5=0$ and $x_4+x_5=0$. Then
  $|S_2|\le 3$, and we conclude $s_3(\mathbf{x})=|S_0\cup t_0\cup
  S_1|\le 6+3+3=12$. Thus, the result is proved for $n=7$.

Assume now $n\equiv 1\pmod{3}$ and $n\ge 10$. Let $\mathbf{x}'$ be the
sequence obtained from $\mathbf{x}$ by removing the first four
coordinates and $\mathbf{x}''$ the sequence obtained from $\mathbf{x}$ by removing the first seven coordinates.
 The lengths of $\mathbf{x}'$ and $\mathbf{x}''$ are multiples of $3$, so
$s_3(\mathbf{x}')\le 2(n-4)-2=2n-10$ and $s_3(\mathbf{x}'')\le 2(n-7)-2=2n-16$.

If $|t_0|\le 2$, then $s_3(\mathbf{x})=|S_0\cup t_0|+s_3(\mathbf{x}')\le 6+2+2n-10=2n-2$.
If $|t_0|=3$ and $|t_1|\le 2$, the condition $|t_0|=3$ implies that $|S_1|\le 3$ and we have   
$s_3(\mathbf{x})=|S_0\cup t_0\cup S_1\cup t_1|+s_3(\mathbf{x}'')
\le 6+3+3+2+2n-16=2n-2$.
Finally, if $|t_0|=|t_1|=3$, it is easy to check that $|S_1|=0$, and therefore,
$s_3(\mathbf{x})=|S_0\cup t_0\cup S_1\cup t_1|+s_3(\mathbf{x}'')\le 6+3+0+3+2n-16=2n-4<2n-2$.

If $n\equiv 2\pmod{3}$, the argument is similar to the case $n\equiv 1\pmod 3$. The only difference now is that
 $S_0$ is formed by the three first rows of $T(x_0x_1x_2x_3x_4)$ and $\mathbf{x}'$ and $\mathbf{x}''$ are obtained from 
$\mathbf{x}$ by removing its first 5 and 8 coordinates, respectively. The maximum value of $|S_0|$ is now 8, but 
this graeter value is compensated by the lower ones  $s_3(\mathbf{x}')\le 2(n-5)-1=2n-12$ 
and $s_3(\mathbf{x}'')\le 2(n-8)-2=2n-18$. \qed
\end{proof}

For $n=4$, the triangles generated by the sequences $\mathbf{h}_1=1001$, $\mathbf{h}_2=r(\mathbf{h})=1110$ and
$\mathbf{h}_3=\ell(\mathbf{h}_1)=0111$ have weight $|T(\mathbf{h}_i)|=|T(\mathbf{z}_2)|=6$, and they do not belong to $Z$.
Proposition~\ref{auxWm} shows that for $n\ge 5$, the triangles generated by sequences
not in $Z$ have weight smaller than those generated by sequences in $Z$.

\begin{prop}
\label{auxWm}
Let $n\ge 5$ be an integer. If $\mathbf{x}\in\mathbb{F}_2^n\setminus Z$ and
$\mathbf{z}\in Z$, then $|T(\mathbf{x})|<|T(\mathbf{z})|$.
\end{prop}
\begin{proof}
For $n\in\{5,6,7\}$ the result is a direct verification. Let $n\ge 8$ and assume that the result holds for
lower values of $n$ (and bigger than 4).

If $\partial^3\mathbf{x}\notin Z^{(n-3)}$, by induction hypothesis $|T(\partial^3\mathbf{x})|<|T(\partial^3\mathbf{z})|$,
and by Proposition~\ref{s3x},
$s_3(\mathbf{x})\le 2n-2=s_3(\mathbf{z})$. Then,
$$
|T(\mathbf{x})|=s_3(\mathbf{x})+|T(\partial^3\mathbf{x})|
\le s_3(\mathbf{z})+|T(\partial^3\mathbf{x})|
<s_3(\mathbf{z})+|T(\partial^3\mathbf{z})|
=|T(\mathbf{z})|.
$$
\indent Now assume that
$\partial^3\mathbf{x}=\mathbf{z}_j^{(n-3)}\in Z^{(n-3)}$. There exist
24 sequences $\mathbf{x}$ such that $\partial^3\mathbf{x}\in
Z^{(n-3)}$.  Table~\ref{p3zi} shows the 6 possible values of
$\partial^2\mathbf{x}$, the 12 possible values of
$\partial\mathbf{x}$, and the 24 possible sequences $\mathbf{x}$. However,
three of these 24 sequences belong to $Z$, so we consider only the
remaining 21, which we have numbered
$\mathbf{x}_1,\ldots,\mathbf{x}_{21}$. Now, Table~\ref{upperBounds}
shows, for each $\mathbf{x}_i$, $i\in[21]$, upper bounds of
$|\partial^2\mathbf{x}_i|$, $|\partial\mathbf{x}_i|$, and
$|\mathbf{x}_i|$. We see that, $s_3(\mathbf{x}_i)<2n-3$ for $n\ge 8$
(the inequality for $\mathbf{x}_{14}$ is $(5n-1)/3<2n-3$, which is
satisfied only for $n\ge 9$, but it is routine to check that
$s_3(\mathbf{x}_{14}^{(8)})=12<13=2\cdot 8-3$).  Then,
\begin{align*}
|T(\mathbf{x}_i)|=& s_3(\mathbf{x}_i)+|T(\partial^3\mathbf{x}_i)|\\
               <& 2n-3+|T(\mathbf{z}_j^{(n-3)})| \\
               =& 2n-2+|T(\mathbf{z}_j^{(n-3)})|-1\\
             \le& s_3(\mathbf{z}_2)+|T(\mathbf{z}_2^{(n-3)})|\\
              = &|T(\mathbf{z}_2)|\le |T(\mathbf{z})|,
\end{align*}
for all $\mathbf{z}\in Z$. \qed
\end{proof}

\begin{thm}
\label{Wm}
Let $n\ge 2$ be an integer. The maximum weight of the triangles of $S(n)$ is
$w_m=\lceil n(n+1)/3\rceil$. Moreover,
\vspace{-\parskip}
\begin{enumerate}[\rm(i)]
\item if $n\equiv 0,2\pmod{3}$, then $W_m=\{\mathbf{z_1}, \,\mathbf{z_2},\,\mathbf{z_3}\}$;
\item if $n\equiv 1\pmod{3}$, then $W_m=\{\mathbf{z_1},\,\mathbf{z_3}\}$.
\end{enumerate}
\end{thm}
\begin{proof}
According to Proposition~\ref{casos n<=4}, the result is true for $n\in\{2,3,4\}$. Let $n\ge 5$.
From Proposition~\ref{auxWm}, the triangles of maximum weight are triangles in the set
$\{T(\mathbf{z}_1),\, T(\mathbf{z}_2),\, T(\mathbf{z}_3)\}$.
Proposition~\ref{pesos T(z)}, gives the weight of triangles in this set:

If $n\equiv 0,2 \pmod{3}$, then
$$
|T(\mathbf{z}_1)|=|T(\mathbf{z}_2)|=|T(\mathbf{z}_3)|=n(n+1)/3=\lceil n(n+1)/3\rceil.
$$
If $n\equiv 1\pmod{3}$, then
$$
|T(\mathbf{z}_1)|=|T(\mathbf{z}_3)|=(n-1)(n+2)/3+1=n(n+1)/3+1/3=\lceil n(n+1)/3\rceil.
$$
and
$$
|T(\mathbf{z}_2)|=(n-1)(n+2)/3=\lceil n(n+1)/3\rceil-1.
$$
Thus, triangles $T(\mathbf{z}_1)$,  $T(\mathbf{z}_2)$ and $T(\mathbf{z}_3)$
give maximum weight $\lceil n(n+1)/3\rceil$, except if $n\equiv 1\pmod{3}$; in this case, only
$T(\mathbf{z}_1)$ and $T(\mathbf{z}_3)$, give the maximum weight because $T(\mathbf{z}_2)$ has weight
one unity less. \qed
\end{proof}

\section{$m-1$ sequences}

For  $4\le n\le 9$, a (computer-aided) exhaustive search gives the values of $w_{m-1}$ and the cardinality of $W_{m-1}$
in the Table~\ref{tmmu}.

\begin{d2}[htb]
$$
\begin{array}{|r|l|l|l|l|l|l|l|l|}
\hline
n         & 4 & 5 & 6 & 7 & 8 & 9 \\
\hline
m        & 4 & 6 & 5 & 12 & 13 & 17 \\
w_{m-1}   & 6 & 9 & 12 & 18 & 22 & 27\\
|W_{m-1}| & 4 & 7 & 30 & 1  & 9 &  22\\
\hline
\end{array}
$$
\caption{Values of $w_{m-1}$ and cardinality of $W_{m-1}$ for  $4\le n\le 9$.}
\label{tmmu}
\end{d2}

Moreover, the corresponding sequences in $W_{m-1}$ are given in Table~\ref{tWmmu}. We see that, for $n=4,5,6,7,8$ and $9$,
the number of non-equivalent sequences in $W_{m-1}$ is $2,2,7,1,2$ and $5$, respectively.

The case $n=7$ can be derived from the results in the previous section, as well as all cases when $n\equiv 1\pmod{3}$,
as shown in the following Theorem.

\begin{thm}
\label{Wmmu1}
Let $n\ge 7$ and $n\equiv 1\pmod{3}$. Then $w_{m-1}=\lceil n(n+1)/3\rceil-1=(n^2+n-2)/3$ and $W_{m-1}=\{\mathbf{z}_2\}$.
\end{thm}
\begin{proof}
By Proposition~\ref{auxWm}, if $n\ge 5$ and $\mathbf{z}\in Z$ and $\mathbf{x}\in\mathbb{F}_2^n\setminus Z$, then
$|T(\mathbf{x}|<|T(\mathbf{z})|$. By Theorem~\ref{Wm}, $W_m=\{\mathbf{z}_1,\,\mathbf{z}_3\}$. Hence,
$W_{m-1}=\{\mathbf{z}_2\}$. Moreover, $w_{m-1}=|T(\mathbf{z}_2)|=\lceil n(n+1)/3\rceil-1=(n^2+n-2)/3$.\qed
\end{proof}

For $n\ge 11$, we define the following sequences. If
$n\equiv 2\pmod{3}$, we consider the sequences equivalent to $\mathbf{v}_1=\overline{100}[n]$ (see Figure~\ref{fvi}):
$$
\begin{array}{lll}
\mathbf{v}_1=\overline{100}[n], &
\mathbf{v}_2=r(\mathbf{v}_1)=0\cdot\overline{101}[n-1], &
\mathbf{v}_3=\ell(\mathbf{v}_1)=\overline{101}[n-1]\cdot 1,  \\
\mathbf{v}_4=i(\mathbf{v}_1)=\overline{010}[n], &
\mathbf{v}_5=r(\mathbf{v}_4)=1\cdot\overline{110}[n-1], &
\mathbf{v}_6=\ell(\mathbf{v}_4)=\overline{110}[n-1]\cdot 0.
\end{array}
$$

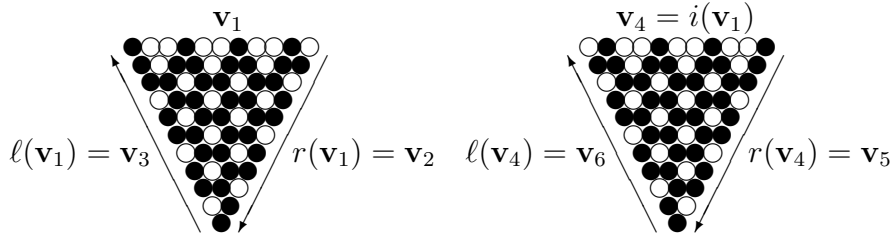
\begin{figure}[htb]
\centering
\setlength{\unitlength}{2.35mm}
\hspace{5mm}
\begin{picture}(14,11)
\y{0.5}{10.5} \n{1.5}{10.5} \n{2.5}{10.5} \y{3.5}{10.5} \n{4.5}{10.5} \n{5.5}{10.5} \y{6.5}{10.5} \n{7.5}{10.5} \n{8.5}{10.5} \y{9.5}{10.5} \n{10.5}{10.5}
\y{1.0}{9.5} \n{2.0}{9.5} \y{3.0}{9.5} \y{4.0}{9.5} \n{5.0}{9.5} \y{6.0}{9.5} \y{7.0}{9.5} \n{8.0}{9.5} \y{9.0}{9.5} \y{10.0}{9.5}
\y{1.5}{8.5} \y{2.5}{8.5} \n{3.5}{8.5} \y{4.5}{8.5} \y{5.5}{8.5} \n{6.5}{8.5} \y{7.5}{8.5} \y{8.5}{8.5} \n{9.5}{8.5}
\n{2.0}{7.5} \y{3.0}{7.5} \y{4.0}{7.5} \n{5.0}{7.5} \y{6.0}{7.5} \y{7.0}{7.5} \n{8.0}{7.5} \y{9.0}{7.5}
\y{2.5}{6.5} \n{3.5}{6.5} \y{4.5}{6.5} \y{5.5}{6.5} \n{6.5}{6.5} \y{7.5}{6.5} \y{8.5}{6.5}
\y{3.0}{5.5} \y{4.0}{5.5} \n{5.0}{5.5} \y{6.0}{5.5} \y{7.0}{5.5} \n{8.0}{5.5}
\n{3.5}{4.5} \y{4.5}{4.5} \y{5.5}{4.5} \n{6.5}{4.5} \y{7.5}{4.5}
\y{4.0}{3.5} \n{5.0}{3.5} \y{6.0}{3.5} \y{7.0}{3.5}
\y{4.5}{2.5} \y{5.5}{2.5} \n{6.5}{2.5}
\n{5.0}{1.5} \y{6.0}{1.5}
\y{5.5}{0.5}
\put(11.5,10.0){\vector(-1,-2){5}}
\put(4.3,0){\vector(-1,2){5}}
\put(5.0,11.7){$\mathbf{v}_1$}
\put(9.5,4.0){$r(\mathbf{v}_1)=\mathbf{v}_2$}
\put(-6.5,4.0){$\ell(\mathbf{v}_1)=\mathbf{v}_3$}
\end{picture}
\hspace{2.5truecm}
\begin{picture}(14,12)
\n{0.5}{10.5} \y{1.5}{10.5} \n{2.5}{10.5} \n{3.5}{10.5} \y{4.5}{10.5} \n{5.5}{10.5} \n{6.5}{10.5} \y{7.5}{10.5} \n{8.5}{10.5}
\n{9.5}{10.5} \y{10.5}{10.5}
\y{1.0}{9.5} \y{2.0}{9.5} \n{3.0}{9.5} \y{4.0}{9.5} \y{5.0}{9.5} \n{6.0}{9.5} \y{7.0}{9.5} \y{8.0}{9.5} \n{9.0}{9.5} \y{10.0}{9.5}
\n{1.5}{8.5} \y{2.5}{8.5} \y{3.5}{8.5} \n{4.5}{8.5} \y{5.5}{8.5} \y{6.5}{8.5} \n{7.5}{8.5} \y{8.5}{8.5} \y{9.5}{8.5}
\y{2.0}{7.5} \n{3.0}{7.5} \y{4.0}{7.5} \y{5.0}{7.5} \n{6.0}{7.5} \y{7.0}{7.5} \y{8.0}{7.5} \n{9.0}{7.5}
\y{2.5}{6.5} \y{3.5}{6.5} \n{4.5}{6.5} \y{5.5}{6.5} \y{6.5}{6.5} \n{7.5}{6.5} \y{8.5}{6.5}
\n{3.0}{5.5} \y{4.0}{5.5} \y{5.0}{5.5} \n{6.0}{5.5} \y{7.0}{5.5} \y{8.0}{5.5}
\y{3.5}{4.5} \n{4.5}{4.5} \y{5.5}{4.5} \y{6.5}{4.5} \n{7.5}{4.5}
\y{4.0}{3.5} \y{5.0}{3.5} \n{6.0}{3.5} \y{7.0}{3.5}
\n{4.5}{2.5} \y{5.5}{2.5} \y{6.5}{2.5}
\y{5.0}{1.5} \n{6.0}{1.5}
\y{5.5}{0.5}
\put(11.5,10.0){\vector(-1,-2){5}}
\put(4.3,0){\vector(-1,2){5}}
\put(2.0,11.7){$\mathbf{v}_4=i(\mathbf{v}_1)$}
\put(9.5,4.0){$r(\mathbf{v}_4)=\mathbf{v}_5$}
\put(-6.5,4.0){$\ell(\mathbf{v}_4)=\mathbf{v}_6$}
\end{picture}
\caption{The sequences $\mathbf{v}_i$ for $n=11$.}
\label{fvi}
\end{figure}

For $n\equiv 0\pmod{3}$, we consider $\mathbf{u}_1=\overline{100}[n]$ and  $\mathbf{u}_7=\overline{010}[n]$.
As $i(\mathbf{u}_7)=\mathbf{u}_7$, the equivalence class of $\mathbf{u}_7$ has only $3$ sequences. We define
$$
\begin{array}{lll}
\mathbf{u}_1=\overline{100}[n],           &
\mathbf{u}_2=r(\mathbf{u}_1)=0\cdot\overline{011}[n-1],   &
\mathbf{u}_3=\ell(\mathbf{u}_1)=\overline{110}[n-1]\cdot 1, \\
\mathbf{u}_4=i(\mathbf{u}_1)=\overline{001}[n],           &
\mathbf{u}_5=r(\mathbf{u}_4)=1\cdot\overline{110}[n-1],   &
\mathbf{u}_6=\ell(\mathbf{u}_4)=\overline{101}[n-1]\cdot 0, \\
\mathbf{u}_7=\overline{010}[n],           &
\mathbf{u}_8=r(\mathbf{u}_7)=0\cdot\overline{101}[n-1],   &
\mathbf{u}_9=\ell(\mathbf{u}_7)=\overline{011}[n-1]\cdot 0.
\end{array}
$$

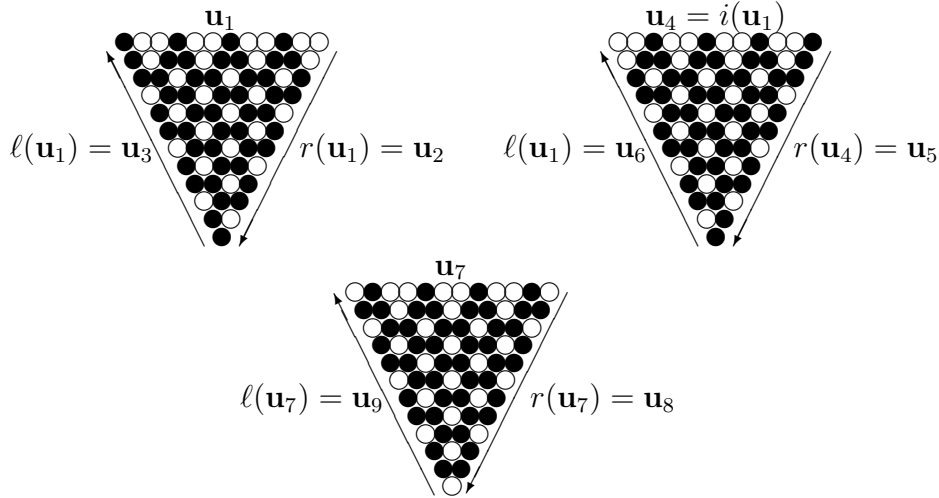
\begin{figure}[htb]
\centering
\setlength{\unitlength}{2.35mm}
\hspace{3mm}
\begin{picture}(14,14)
\y{0.5}{11.5} \n{1.5}{11.5} \n{2.5}{11.5} \y{3.5}{11.5} \n{4.5}{11.5} \n{5.5}{11.5} \y{6.5}{11.5} \n{7.5}{11.5} \n{8.5}{11.5} \y{9.5}{11.5} \n{10.5}{11.5} \n{11.5}{11.5}
\y{1.0}{10.5} \n{2.0}{10.5} \y{3.0}{10.5} \y{4.0}{10.5} \n{5.0}{10.5} \y{6.0}{10.5} \y{7.0}{10.5} \n{8.0}{10.5} \y{9.0}{10.5} \y{10.0}{10.5} \n{11.0}{10.5}
\y{1.5}{9.5} \y{2.5}{9.5} \n{3.5}{9.5} \y{4.5}{9.5} \y{5.5}{9.5} \n{6.5}{9.5} \y{7.5}{9.5} \y{8.5}{9.5} \n{9.5}{9.5} \y{10.5}{9.5}
\n{2.0}{8.5} \y{3.0}{8.5} \y{4.0}{8.5} \n{5.0}{8.5} \y{6.0}{8.5} \y{7.0}{8.5} \n{8.0}{8.5} \y{9.0}{8.5} \y{10.0}{8.5}
\y{2.5}{7.5} \n{3.5}{7.5} \y{4.5}{7.5} \y{5.5}{7.5} \n{6.5}{7.5} \y{7.5}{7.5} \y{8.5}{7.5} \n{9.5}{7.5}
\y{3.0}{6.5} \y{4.0}{6.5} \n{5.0}{6.5} \y{6.0}{6.5} \y{7.0}{6.5} \n{8.0}{6.5} \y{9.0}{6.5}
\n{3.5}{5.5} \y{4.5}{5.5} \y{5.5}{5.5} \n{6.5}{5.5} \y{7.5}{5.5} \y{8.5}{5.5}
\y{4.0}{4.5} \n{5.0}{4.5} \y{6.0}{4.5} \y{7.0}{4.5} \n{8.0}{4.5}
\y{4.5}{3.5} \y{5.5}{3.5} \n{6.5}{3.5} \y{7.5}{3.5}
\n{5.0}{2.5} \y{6.0}{2.5} \y{7.0}{2.5}
\y{5.5}{1.5} \n{6.5}{1.5}
\y{6.0}{0.5}
\put(12.5,11.0){\vector(-1,-2){5.5}}
\put(5.0,0){\vector(-1,2){5.5}}
\put(5.0,12.5){$\mathbf{u}_1$}
\put(10.4,5.0){$r(\mathbf{u}_1)=\mathbf{u}_2$}
\put(-6.0,5.0){$\ell(\mathbf{u}_1)=\mathbf{u}_3$}
\end{picture}
\hspace{3truecm}
\begin{picture}(14,14)
\n{0.5}{11.5} \n{1.5}{11.5} \y{2.5}{11.5} \n{3.5}{11.5} \n{4.5}{11.5} \y{5.5}{11.5} \n{6.5}{11.5} \n{7.5}{11.5} \y{8.5}{11.5} \n{9.5}{11.5} \n{10.5}{11.5} \y{11.5}{11.5}
\n{1.0}{10.5} \y{2.0}{10.5} \y{3.0}{10.5} \n{4.0}{10.5} \y{5.0}{10.5} \y{6.0}{10.5} \n{7.0}{10.5} \y{8.0}{10.5} \y{9.0}{10.5} \n{10.0}{10.5} \y{11.0}{10.5}
\y{1.5}{9.5} \n{2.5}{9.5} \y{3.5}{9.5} \y{4.5}{9.5} \n{5.5}{9.5} \y{6.5}{9.5} \y{7.5}{9.5} \n{8.5}{9.5} \y{9.5}{9.5} \y{10.5}{9.5}
\y{2.0}{8.5} \y{3.0}{8.5} \n{4.0}{8.5} \y{5.0}{8.5} \y{6.0}{8.5} \n{7.0}{8.5} \y{8.0}{8.5} \y{9.0}{8.5} \n{10.0}{8.5}
\n{2.5}{7.5} \y{3.5}{7.5} \y{4.5}{7.5} \n{5.5}{7.5} \y{6.5}{7.5} \y{7.5}{7.5} \n{8.5}{7.5} \y{9.5}{7.5}
\y{3.0}{6.5} \n{4.0}{6.5} \y{5.0}{6.5} \y{6.0}{6.5} \n{7.0}{6.5} \y{8.0}{6.5} \y{9.0}{6.5}
\y{3.5}{5.5} \y{4.5}{5.5} \n{5.5}{5.5} \y{6.5}{5.5} \y{7.5}{5.5} \n{8.5}{5.5}
\n{4.0}{4.5} \y{5.0}{4.5} \y{6.0}{4.5} \n{7.0}{4.5} \y{8.0}{4.5}
\y{4.5}{3.5} \n{5.5}{3.5} \y{6.5}{3.5} \y{7.5}{3.5}
\y{5.0}{2.5} \y{6.0}{2.5} \n{7.0}{2.5}
\n{5.5}{1.5} \y{6.5}{1.5}
\y{6.0}{0.5}
\put(12.5,11.0){\vector(-1,-2){5.5}}
\put(5.0,0){\vector(-1,2){5.5}}
\put(2.0,12.5){$\mathbf{u}_4=i(\mathbf{u}_1)$}
\put(10.4,5.0){$r(\mathbf{u}_4)=\mathbf{u}_5$}
\put(-6.0,5.0){$\ell(\mathbf{u}_1)=\mathbf{u}_6$}
\end{picture}
\\
\begin{picture}(14,14)
\n{0.5}{11.5} \y{1.5}{11.5} \n{2.5}{11.5} \n{3.5}{11.5} \y{4.5}{11.5} \n{5.5}{11.5} \n{6.5}{11.5} \y{7.5}{11.5} \n{8.5}{11.5} \n{9.5}{11.5} \y{10.5}{11.5} \n{11.5}{11.5}
\y{1.0}{10.5} \y{2.0}{10.5} \n{3.0}{10.5} \y{4.0}{10.5} \y{5.0}{10.5} \n{6.0}{10.5} \y{7.0}{10.5} \y{8.0}{10.5} \n{9.0}{10.5} \y{10.0}{10.5} \y{11.0}{10.5}
\n{1.5}{9.5} \y{2.5}{9.5} \y{3.5}{9.5} \n{4.5}{9.5} \y{5.5}{9.5} \y{6.5}{9.5} \n{7.5}{9.5} \y{8.5}{9.5} \y{9.5}{9.5} \n{10.5}{9.5}
\y{2.0}{8.5} \n{3.0}{8.5} \y{4.0}{8.5} \y{5.0}{8.5} \n{6.0}{8.5} \y{7.0}{8.5} \y{8.0}{8.5} \n{9.0}{8.5} \y{10.0}{8.5}
\y{2.5}{7.5} \y{3.5}{7.5} \n{4.5}{7.5} \y{5.5}{7.5} \y{6.5}{7.5} \n{7.5}{7.5} \y{8.5}{7.5} \y{9.5}{7.5}
\n{3.0}{6.5} \y{4.0}{6.5} \y{5.0}{6.5} \n{6.0}{6.5} \y{7.0}{6.5} \y{8.0}{6.5} \n{9.0}{6.5}
\y{3.5}{5.5} \n{4.5}{5.5} \y{5.5}{5.5} \y{6.5}{5.5} \n{7.5}{5.5} \y{8.5}{5.5}
\y{4.0}{4.5} \y{5.0}{4.5} \n{6.0}{4.5} \y{7.0}{4.5} \y{8.0}{4.5}
\n{4.5}{3.5} \y{5.5}{3.5} \y{6.5}{3.5} \n{7.5}{3.5}
\y{5.0}{2.5} \n{6.0}{2.5} \y{7.0}{2.5}
\y{5.5}{1.5} \y{6.5}{1.5}
\n{6.0}{0.5}
\put(12.5,11.5){\vector(-1,-2){5.7}}
\put(5.0,0){\vector(-1,2){5.7}}
\put(5.0,12.5){$\mathbf{u}_7$}
\put(10.4,5.0){$r(\mathbf{u}_7)=\mathbf{u}_8$}
\put(-6.0,5.0){$\ell(\mathbf{u}_7)=\mathbf{u}_9$}
\end{picture}
\caption{The sequences $\mathbf{u}_i$ for $n=12$.}
\label{fui}
\end{figure}

\begin{remark}\normalfont
\label{rem casos 11 i 12}
If $n\equiv 2\pmod{3}$, the primitives of $\mathbf{z}_3^{(n-1)}$ are $\mathbf{z}_2$ and
$\mathbf{1}+\mathbf{z}_2=\overline{100}[n]=\mathbf{v}_1$, and the
primitives of $\mathbf{z}_1^{(n-1)}$ are $\mathbf{z}_3$ and
$\mathbf{1}+\mathbf{z}_3=\overline{010}[n]=\mathbf{v}_4$. We see that
$V=V^{(n)}=\{\mathbf{v}_i: i\in[9]\}$ is the union of the equivalence classes of the sequences
$\mathbf{x}\notin W_m$ such that $\partial\mathbf{x}\in W_m^{(n-1)}$.

In $n\equiv 0\pmod{3}$, the primitives of $\mathbf{z}_1^{(n-1)}$, $\mathbf{z}_2^{(n-1)}$ and $\mathbf{z}_3^{(n-1)}$ which are not in
$W_m$ are $\mathbf{1}+\mathbf{z}_3=\overline{010}[n]=\mathbf{u}_7$,
$\mathbf{1}+\mathbf{z}_1=\overline{001}[n]=\mathbf{u}_4$ and $\mathbf{1}+\mathbf{z}_2=\overline{100}[n]=\mathbf{u}_1$, 
respectively. We see that $U=U^{(n)}=\{\mathbf{u}_i:i\in[9]\}$ is the union of the equivalence classes of the sequences
$\mathbf{x}\notin W_m$ such that $\partial\mathbf{x}\in W_m^{(n-1)}$.
\end{remark}

\begin{remark}\normalfont
\label{pesos u iv}
The weight of the sequences $\mathbf{u}_i$ and $\mathbf{v}_i$ are the following.
$$
\begin{array}{lll}
|\mathbf{u}_1|=n/3,     & |\mathbf{u}_2|=(2n-3)/3, & |\mathbf{u}_3|=(2n+3)/3, \\
|\mathbf{u}_4|=n/3,     & |\mathbf{u}_5|=(2n+3)/3, & |\mathbf{u}_6|=(2n-3)/3, \\
|\mathbf{u}_7|=n/3,     & |\mathbf{u}_8|=(2n-3)/3, &|\mathbf{u}_9|=(2n-3)/3.  \\[4pt]
|\mathbf{v}_1|=(n+1)/3, & |\mathbf{v}_2|=(2n-1)/3, & |\mathbf{v}_3|=(2n+2)/3, \\
|\mathbf{v}_4|=(n+1)/3, &|\mathbf{v}_5|=(2n+2)/3,  & |\mathbf{u}_6|=(2n-1)/3.
\end{array}
$$
\end{remark}

Table~\ref{tWmmu casos 11 i 12} gives the values of $w_{m-1}$ and the set $W_{m-1}$ for $n\in\{11,12\}$.

\begin{d2}[H]
$$
\begin{array}{|r|r|l|}
\hline
n & w_{m-1} & W_{m-1} \\
\hline
11 & 41 & \{\mathbf{v}_1,\,
         \mathbf{v}_2,\,
         \mathbf{v}_3,\,
         \mathbf{v}_4,\,
         \mathbf{v}_5,\,
         \mathbf{v}_6\} \\
\hline
12 & 48 & \{\mathbf{u}_1,\,
         \mathbf{u}_2,\,
         \mathbf{u}_3,\,
        \mathbf{u}_4,\,
         \mathbf{u}_5,\,
         \mathbf{u}_6,\, \\
   &    & \phantom{\{}\mathbf{u}_7,\,
         \mathbf{u}_8,\,
         \mathbf{u}_9\}\\
\hline
\end{array}
$$
\caption{Values $w_{m-1}$ and sets $W_{m-1}$ for $n\in\{11,12\}$.}
\label{tWmmu casos 11 i 12}
\end{d2}

Let $Y=U$ if $n\equiv 0\pmod{3}$, and $Y=V$ if $n\equiv 2\pmod{3}$. 
\begin{prop}
\label{Wmmu}
Let $n\ge 11$ be an integer and $n\equiv 0,2\pmod{3}$. If $\mathbf{y}\in Y$, then 
$|T(\mathbf{y})|=\lceil n^2/3\rceil$.
\end{prop}
\begin{proof}
 Consider first the case $n\equiv 2\pmod{3}$. Because $Y=\{\mathbf{v}_i: i\in[6]\}$ is an equivalence class, all
    sequences $\mathbf{v}_i$ generate triangles of the same
    weight. Then, it is sufficient to obtain $|T(\mathbf{v}_1)|$. We
    have $\mathbf{v}_1=\overline{100}[n]$, so
    $\partial\mathbf{v}_1=\overline{101}[n-1]=\mathbf{z}_3^{(n-1)}$. Then,
$$
|T(\mathbf{v}_1)|=|\mathbf{v}_1|+|T(\mathbf{z}_3^{(n-1)})|
=(n+1)/3+\lceil (n-1)n/3\rceil
=\lceil (n^2+1)/3\rceil=\lceil n^2/3\rceil.
$$
Now assume  $n\equiv 0 \pmod{3}$. As before, $\mathbf{u}_1=\overline{100}[n]$, $\partial\mathbf{u}_1=\mathbf{z}_3^{(n-1)}$ and
$$
|T(\mathbf{u}_1)|=|\mathbf{u}_1|+|T(\mathbf{z}_3^{(n-1)})|
=\lceil n^2/3\rceil. 
$$
The sequences $\mathbf{u}_i$ for $i\in[6]$ are equivalent to $\mathbf{u_1}$. Therefore, 
$T(\mathbf{u}_i)|=|T(\mathbf{u}_1)|=\lceil n^2/3\rceil$.
Finally, consider $\mathbf{u}_7=\overline{010}[n]$. We have $\partial\mathbf{u}_7=\overline{110}[n-1]=\mathbf{z}_1^{(n-1)}$.
Then,
$$
|T(\mathbf{u}_7)|=|\mathbf{u}_7|+|T(\mathbf{z}_1^{(n-1)})|
=n/3+\lceil (n-1)n/3\rceil
=\lceil n^2/3\rceil,
$$
and $|T(\mathbf{u}_i)|=\lceil n^2/3\rceil$ for $i\in\{7,8,9\}$.\qed
\end{proof}

\begin{remark}\normalfont
Note that $\lceil n^2/3\rceil=\lceil n(n+1)/3\rceil-\lfloor n/3\rfloor=w_m-\lfloor n/3\rfloor$.
\end{remark}
 
Computational experimentation suggests the following conjecture.

\noindent\textbf{Conjecture}
If $n\ge 11$ and $n\equiv 0,2\pmod{3}$, then $W_{m-1}=Y$, 
and $w_{m-1}=\lceil n^2/3\rceil$.

\section{Large tables}

\begin{d2}[H]
$$
\begin{array}{|l|l|l|l|}
\hline
\rule{0pt}{3ex}
\partial^2\mathbf{x} & \partial\mathbf{x} & \mathbf{x} & |\mathbf{x}|+|\partial\mathbf{x}|\ge \\
\hline
\rule{0pt}{3ex}
%%%a1
\mathbf{a}_1=\mathbf{1}^{(n-2)} & \overline{01}[n-1]
& \overline{0011}[n]=\mathbf{c}_1 & \\
&&\overline{1100}[n]=\mathbf{c}_4 &  \\
                                &  \overline{10}[n-1]
&  0\cdot\overline{1100}[n-1] & n/2+n/2 \\
&& 1\cdot\overline{0011}[n-1] & n/2+n/2\\
\hline
\rule{0pt}{3ex}
%%%a2
\mathbf{a}_2=1\cdot 0[n-3] & 0\cdot \overline{1}[n-2]
&  0\cdot\overline{01}[n-1]  &  (n-2)/2+(n-2) \\
&& 1\cdot\overline{10}[n-1]  &  (n+2)/2+(n-2)  \\
& 1\cdot\overline{0}[n-2]
 & 0\cdot \overline{1}[n-1]  & (n-1)+1\\
&&  1\cdot \overline{0}[n-1]=\mathbf{a}_2   & \\
\hline
\rule{0pt}{3ex}
%%%a3
\mathbf{a}_3=\overline{0}[n-3]\cdot 1 & \overline{0}[n-2]\cdot 1
& \overline{0}[n-1]\cdot 1 =\mathbf{a}_3 & \\
&& \overline{1}[n-1]\cdot 0 & (n-1)+1 \\
& \overline{1}[n-2]\cdot 0
&\overline{01}[n-1]\cdot 0 &  (n-2)/2+(n-2)\\
&&\overline{10}[n-1]\cdot 1 & (n+2)/2+(n-2)\\
\hline
\end{array}
$$
\caption{Sequences $\mathbf{x}$ such that $\partial^2\mathbf{x}\in W_1$ ($n\ge 12$ even).}
\label{d2x=ai}
\end{d2}

\begin{d2}[H]
$$
\begin{array}{|l|l|l|l|}
\hline
\rule{0pt}{3ex}
\partial^2\mathbf{x} & \partial\mathbf{x} & \mathbf{x} & |\mathbf{x}|+|\partial\mathbf{x}|\ge \\
\hline
\rule{0pt}{3ex}
%%%b1
\mathbf{b}_1=\overline{10}[n-2] & \overline{0110}[n-1]  & \overline{0010}[n] & (n-2)/4+(n-2)/2 \\
                                &                       & \overline{1101}[n] & 3n/4+(n-2)/2  \\
                                & \overline{1001}[n-1]  & \overline{0111}[n] & (3n-2)/4+(n-2)/2  \\
                                &                       & \overline{1000}[n] & n/4+(n-2)/2 \\
\hline
\rule{0pt}{3ex}
%%%b2
\mathbf{b}_2= 01\cdot\overline{0}[n-4] & 00\cdot\overline{1}[n-3] & 00\cdot\overline{01}[n-2] & (n-2)/2+(n-3)\\
                                       &                        & 11\cdot\overline{10}[n-2] &  (n+2)/2+(n-3)\\
                                       & 11\cdot\overline{0}[n-3] & 01\cdot\overline{0}[n-2]=\mathbf{b}_2 & \\
                                       &                          & 10\cdot\overline{1}[n-2] & (n-1)+2\\
\hline
\rule{0pt}{3ex}
%%%b3
\mathbf{b}_3=\overline{0}[n-4]\cdot 11 & \overline{0}[n-3]\cdot 10 & \overline{0}[n-2]\cdot 11=\mathbf{b}_3  & \\
                                            &                               & \overline{1}[n-2]\cdot 00 & (n-2)+1\\
                                            & \overline{1}[n-3]\cdot 01 & \overline{01}[n-2]\cdot 10   & n/2+(n-2)\\
                                            &                                & \overline{10}[n-2]\cdot 01 & n/2+(n-2) \\
\hline
\rule{0pt}{3ex}
%%%b4
\mathbf{b}_4= \overline{01}[n-2] & \overline{0011}[n-1] & \overline{0001}[n] & (n-2)/4+(n-2)/2\\
                                 &                      & \overline{1110}[n] & 3n/4+(n-2)/2\\
                                 & \overline{1100}[n-1] & \overline{0100}[n] & n/4+n/2\\
                                 &                      & \overline{1011}[n] & (3n-2)/4+n/2\\
\hline
\rule{0pt}{3ex}
%%%b5
\mathbf{b}_5= 11\cdot\overline{0}[n-4] & 01\cdot\overline{0}[n-3] & 00\cdot\overline{1}[n-2] & (n-2)+1\\
                                       &                          & 11\cdot\overline{0}[n-2]=\mathbf{b}_5 & \\
                                       & 10\cdot\overline{1}[n-3] & 01\cdot\overline{10}[n-2] & n/2+(n-2) \\
                                       &                          & 10\cdot\overline{01}[n-2] & n/2+(n-2)\\
\hline
\rule{0pt}{3ex}
%%%b6
\mathbf{b}_6= \overline{0}[n-4]\cdot 10 & \overline{0}[n-3]\cdot 11 & \overline{0}[n-2]\cdot 10=\mathbf{b}_6 & \\
                                        &                           & \overline{1}[n-2]\cdot 01 & (n-1)+2\\
                                        & \overline{1}[n-3]\cdot 00 & \overline{01}[n-2]\cdot 11 & (n+2)/2+(n-3)\\
                                        &                           & \overline{10}[n-2]\cdot 00 & (n-2)/2+(n-3)\\
\hline
\end{array}
$$
\caption{Sequences $\mathbf{x}$ such that $\partial^2\mathbf{x}\in W_2$ ($n\ge 12$ even).}
\label{d2x=bi}
\end{d2}

\begin{d2}[H]
$$
\begin{array}{|l|l|l|}
\hline
\rule{0pt}{3ex}
\partial^2\mathbf{x} & \partial\mathbf{x} & \mathbf{x}  \\
\hline
\rule{0pt}{3ex}
%c1
\mathbf{c}_1=\overline{0011}[n-2] & \overline{0001}[n-1] & \overline{00001111}[n] \\
                                  &                      & \overline{11110000}[n] \\
                                  & \overline{1110}[n-1] & \overline{01011010}[n]  \\
                                  &                      & \overline{10100101}[n]  \\
\hline
\rule{0pt}{3ex}
%c2
\mathbf{c}_2=101 \cdot\overline{0}[n-5] & 011\cdot\overline{0}[n-4]  & 001 \cdot\overline{0}[n-3]=\mathbf{c}_5  \\
                                        &                            & 110 \cdot\overline{1}[n-3]   \\
                                        & 100\cdot\overline{1}[n-4]  & 011 \cdot\overline{10}[n-3]   \\
                                        &                            & 100 \cdot\overline{01}[n-3]   \\
\hline
\rule{0pt}{3ex}
%c3
\mathbf{c}_3=\overline{0}[n-5]\cdot 100 & \overline{0}[n-4]\cdot 111 & \overline{0}[n-3]\cdot 101=\mathbf{c}_6   \\
                                        &                            & \overline{1}[n-3]\cdot 010   \\
                                        & \overline{1}[n-4]\cdot 000 & \overline{01}[n-4]\cdot 0000  \\
                                        &                            & \overline{10}[n-4]\cdot 1111  \\
\hline
\rule{0pt}{3ex}
%c4
\mathbf{c}_4=\overline{1100}[n-2] & \overline{0100}[n-1] & \overline{00111100}[n]   \\
                                  &                      & \overline{11000011}[n]   \\
                                  & \overline{1011}[n-1] & \overline{01101001}[n]   \\
                                  &                      & \overline{10010110}[n]   \\
\hline
\rule{0pt}{3ex}
%c5
\mathbf{c}_5=001 \cdot\overline{0}[n-5] & 000 \cdot\overline{1}[n-4] & 000 \cdot\overline{01}[n-3]   \\
                                        &                            & 111 \cdot\overline{10}[n-3]   \\
                                        & 111 \cdot\overline{0}[n-4] & 010 \cdot\overline{1}[n-3]   \\
                                        &                            & 101 \cdot\overline{0}[n-3]=\mathbf{c}_2   \\
\hline
\rule{0pt}{3ex}
%c6
\mathbf{c}_6=\overline{0}[n-5]\cdot 101 & \overline{0}[n-4]\cdot 110 & \overline{0}[n-3]\cdot 100=\mathbf{c}_3    \\
                                        &                            & \overline{1}[n-3]\cdot 011   \\
                                        & \overline{1}[n-4]\cdot 001 & \overline{01}[n-3]\cdot 001  \\
                            &                                        & \overline{10}[n-3]\cdot 110   \\
\hline
\end{array}
$$
\caption{Sequences $\mathbf{x}$ such that $\partial^2\mathbf{x}\in C^{(n-2)}$ ($n\ge 12$ even).}
\label{d2x=ci}
\end{d2}

\begin{d2}[H]
$$
\begin{array}{|l|l|l|l|l|}
\hline
\rule{0pt}{3ex}
\partial^3\mathbf{x} & \partial^2\mathbf{x} &  \partial\mathbf{x} & \mathbf{x} & \mathbf{x}_i \\
\hline
\rule{0pt}{3ex}
%z1
\mathbf{z}_1=\overline{110}[n-3]  & \overline{010}[n-2]       & \overline{001110}[n-1] & \overline{000101111010}[n] &\mathbf{x}_1 \\
                               &                              &                        & \overline{111010000101}[n] &\mathbf{x}_2 \\
                               &                              & \overline{110001}[n-1] & \overline{010000101111}[n]      & \mathbf{x}_3 \\
                               &                              &                        & \overline{101111010000}[n]      & \mathbf{x}_4 \\
                               & \overline{101}[n-2]=\mathbf{z}_3  & \overline{011}[n-1]=\mathbf{z}_2  & \overline{001}[n]        & \mathbf{x}_5 \\
                               &                              &                        & \overline{110}[n]=\mathbf{z}_1  &   \\
                               &                              & \overline{100}[n-1]    & \overline{011100}[n]    & \mathbf{x}_6 \\
                               &                              &                        & \overline{100011}[n]    & \mathbf{x}_7 \\
\hline
\rule{0pt}{3ex}
%z2
\mathbf{z}_2=\overline{011}[n-3] & \overline{001}[n-2]        & \overline{000111}[n-1] & \overline{000010111101}[n]      & \mathbf{x}_8  \\
                               &                              &                        & \overline{111101000010}[n]      & \mathbf{x}_9 \\
                               &                              & \overline{111000}[n-1] & \overline{010111101000}[n]      & \mathbf{x}_{10} \\
                               &                              &                        & \overline{101000010111}[n]      & \mathbf{x}_{11} \\
                               & \overline{110}[n-2]=\mathbf{z}_1  & \overline{010}[n-1] & \overline{001110}[n]  & \mathbf{x}_{12} \\
                               &                              &                          & \overline{110001}[n]  & \mathbf{x}_{13} \\
                               &                              & \overline{101}[n-1]=\mathbf{z}_3  &\overline{011}[n]=\mathbf{z}_2 &    \\
                               &                              &                        &\overline{100}[n]        & \mathbf{x}_{14} \\
\hline
\rule{0pt}{3ex}
%z3
\mathbf{z}_3=\overline{101}[n-3]  & \overline{011}[n-2]=\mathbf{z}_2 & \overline{001}[n-1]     & \overline{000111}[n]    & \mathbf{x}_{15} \\
                               &                              &                                 & \overline{111000}[n]   & \mathbf{x}_{16} \\
                               &                              & \overline{110}[n-1]=\mathbf{z}_1 & \overline{010}[n]     & \mathbf{x}_{17} \\
                               &                              &                                 & \overline{101}[n]=\mathbf{z}_3  &  \\
                               & \overline{100}[n-2]          & \overline{011100}[n-1] & \overline{001011110100}[n] & \mathbf{x}_{18} \\
                               &                              &                        & \overline{110100001011}[n] & \mathbf{x}_{19} \\
                               &                              & \overline{100011}[n-1] & \overline{011110100001}[n] & \mathbf{x}_{20} \\
                               &                              &                        & \overline{100001011110}[n] & \mathbf{x}_{21} \\
\hline
\end{array}
$$
\caption{Third primitives of $\mathbf{z}_i^{(n-3)}$.}
\label{p3zi}
\end{d2}

\begin{d2}[H]
$$
\begin{array}{|l|l|l|l|l|}
\hline
i  & |\partial^2\mathbf{x}_i|\le &  |\partial\mathbf{x}_i|\le & |\mathbf{x}_i|\le & s_3(\mathbf{x}_i)\le \\
\hline
%z1
1 & (n-1)/3    & (2n-3)/3 & (n+4)/2 & (9n+4)/6  \\
2 & (n-1)/3    & (2n-3)/3 & (n+4)/2 & (9n+4)/6  \\
3 & (n-1)/3    & (n+1)/2  & (n+4)/2  & (8n+13)/6  \\
4 & (n-1)/3    & (n+1)/2  & (n+4)/2  & (8n+13)/6  \\
5 & (2n-3)/3   & (2n-2)/3 & n/3 & (5n-6)/3   \\
6 & (2n-3)/3   & (n+1)/3  & (2n-2)/3 & (5n-4)/3  \\
7 & (2n-3)/3   & (n+1)/3  & (2n-2)/3 & (5n-4)/3  \\
\hline
%z2
8 & (n-2)/3    & (n-1)/2  & n/2 & (8n-7)/6 \\
9 & (n-2)/3    & (n-1)/2  & (n+4)/2 & (8n+5)/6 \\
10 & (n-2)/3  & (2n-2)/3 & (n+3)/2 & (9n+1)/6 \\
11 & (n-2)/3  & (2n-2)/3 & (n+1)/2 & (9n-5)/6 \\
12 & (2n-2)/3 & n/3  & (n+1)/2 & (9n-1)/6     \\
13 & (2n-2)/3 & n/3  & (n+2)/2 & (9n+2)/6     \\
14 & (2n-2)/3 & (2n-1)/3 & (n+2)/3 & (5n-1)/3 \\
\hline
%z3
15 & (2n-4)/3 & (n-1)/3  & n/2  & (9n-10)/6 \\
16 & (2n-4)/3 & (n-1)/3  & (n+3)/2  & (9n-1)/6 \\
17 & (2n-4)/3 & 2n/3     & (n+1)/3 & (5n-3)/3 \\
18 & n/3      & (n+1)/2      & (n+2)/2  & (8n+9)/6 \\
19 & n/3      & (n+1)/2      & (n+2)/2  & (8n+9)/6 \\
20 & n/3     & n/2 & (n+3)/2  & (8n+9)/6 \\
21 & n/3     & n/2 & (n+3)/2  & (8n+9)/6 \\
\hline
\end{array}
$$
\caption{Upper bounds of $|\partial^2\mathbf{x}_i|$, $|\partial\mathbf{x}_i|$ and $|\mathbf{x}_i|$.}
\label{upperBounds}
\end{d2}

\begin{d2}[H]
$$
\begin{array}{|r|rrr|}
\hline
n & W_{m-1}           &                      &                         \\
\hline
4 & \mathbf{x}_1=1001 & r(\mathbf{x}_1)=1110 &  \ell(\mathbf{x}_1)=0111 \\
  & \mathbf{x}_2=0110 &                      &                          \\
\hline
5 & \mathbf{x}_1=10010    & r(\mathbf{x}_1)=01011    &  \ell(\mathbf{x}_1)=10111   \\
  & i(\mathbf{x}_1)=01001 & r(i(\mathbf{x}_1))=11101 &  \ell(i(\mathbf{x}_1))=11010 \\
  & \mathbf{x}_2=01110    &                          &                             \\
\hline
6 & \mathbf{x}_1=100100    &  r(\mathbf{x}_1)=001101    & \ell(\mathbf{x}_1)=110111    \\
  & i(\mathbf{x}_1)=001001 &  r(i(\mathbf{x}_1))=111011 &  \ell(i(\mathbf{x}_1))=101100 \\
  & \mathbf{x}_2=110100    &  r(\mathbf{x}_2)=001110    & \ell(\mathbf{x}_2)=011101    \\
  & i(\mathbf{x}_2)=001011 &  r(i(\mathbf{x}_2))=101110 &  \ell(i(\mathbf{x}_2))=011100 \\
  & \mathbf{x}_3=010011    &  r(\mathbf{x}_3)=101001    & \ell(\mathbf{x}_3)=111010    \\
  & i(\mathbf{x}_3)=110010 &  r(i(\mathbf{x}_3))=010111 &  \ell(i(\mathbf{x}_3))=100101 \\
  & \mathbf{x}_4=100011    &  r(\mathbf{x}_4)=101011    & \ell(\mathbf{x}_4)=101111    \\
  & i(\mathbf{x}_4)=110001 &  r(i(\mathbf{x}_4))=111101 &  \ell(i(\mathbf{x}_4))=110101 \\
  & \mathbf{x}_5=010010    &  r(\mathbf{x}_5)=010110    & \ell(\mathbf{x}_5)=011010    \\
  & \mathbf{x}_6=100111    &                            &                              \\
  & i(\mathbf{x}_6)=111001 &                            &                              \\
  & \mathbf{x}_7=011110    &                            &                              \\

\hline
7 & \mathbf{x}_1=0110110   &                            &                              \\
\hline
8 &  \mathbf{x}_1=10010010    &  r(\mathbf{x}_1)=01011011    & \ell(\mathbf{x}_1)=10110111    \\
  & i(\mathbf{x}_1)=01001001 &  r(i(\mathbf{x}_1))=11101101 &  \ell(i(\mathbf{x}_1))=11011010 \\
  &  \mathbf{x}_2=11010110    &  r(\mathbf{x})_2=01101011    & \ell(\mathbf{x})_2=10111101    \\
\hline
9 &  \mathbf{x}_1=100100100    &  r(\mathbf{x}_1)=001101101    & \ell(\mathbf{x}_1)=110110111   \\
  & i(\mathbf{x}_1)=001001001 &  r(i(\mathbf{x}_1))=111011011 &  \ell(i(\mathbf{x}_1))=101101100 \\
  &  \mathbf{x}_2=010011011    &  r(\mathbf{x}_2)=101101011    & \ell(\mathbf{x}_2)=101111010   \\
  & i(\mathbf{x}_2)=110110010 &  r(i(\mathbf{x}_2))=010111101 &  \ell(i(\mathbf{x}_2))=110101101 \\
  & \mathbf{x}_3=011100011    &  r(\mathbf{x}_3)=101011011    & \ell(\mathbf{x}_3)=101101110   \\
  & i(\mathbf{x}_3)=110001110 &  r(i(\mathbf{x}_3))=011101101 &  \ell(i(\mathbf{x}_3))=1101100101 \\
  & \mathbf{x}_4=010010010    &  r(\mathbf{x}_4)=010110110    & \ell(\mathbf{x}_4)=011011010   \\
  & \mathbf{x}_5=011010110    &      &    \\
\hline
\end{array}
$$
\caption{The sequences in $W_{m-1}$ for $4\le n\le 9$.}
\label{tWmmu}
\end{d2}

\end{document}